\documentclass[12pt,reqno]{amsart}
\usepackage{amsfonts,amsmath,amssymb,graphicx,array}
\usepackage{verbatim}

\usepackage{tikz}
\usetikzlibrary{arrows}
\usetikzlibrary{fadings}
\usetikzlibrary{decorations.shapes}

\tikzset{decorate sep/.style 2 args=
{decorate,decoration={shape backgrounds,shape=circle,shape size=#1,shape sep=#2}}}

\theoremstyle{plain} \newtheorem{Thm}{Theorem}[section]
\theoremstyle{plain} \newtheorem{Cor}[Thm]{Corollary}
\theoremstyle{plain} \newtheorem{Prop}[Thm]{Proposition}
\theoremstyle{plain} \newtheorem{Lemma}[Thm]{Lemma}
\theoremstyle{definition} 
\theoremstyle{definition} \newtheorem{Rem}[Thm]{Remark}
\theoremstyle{definition} 
\theoremstyle{definition} 

\newcommand{\thmlist}{
\renewcommand{\theenumi}{\alph{enumi}}
\renewcommand{\labelenumi}{(\theenumi)}}
\renewcommand{\Re}{\mathop{\rm{Re}}}
\renewcommand{\Im}{\mathop{\rm{Im}}}

\newcommand{\id}{\mathop{\rm{id}}}

\newcommand{\Ad}{\mathop{\rm{Ad}}}
\newcommand{\ad}{\mathop{\rm{ad}}}

\newcommand{\Ker}{\mathop{\rm{ker}}}

\newcommand{\inner}[2]{\langle#1,#2\rangle}

\newcommand{\C}{\ensuremath{\mathbb C}}

\newcommand{\D}{\ensuremath{\mathbb D}}

\newcommand{\R}{\ensuremath{\mathbb R}}
\newcommand{\Z}{\ensuremath{\mathbb Z}}
\newcommand{\N}{\ensuremath{\mathbb N}}

\renewcommand{\l}{\lambda}
\renewcommand{\a}{\alpha}
\renewcommand{\b}{\beta}

\newcommand{\fa}{\mathfrak{a}}

\newcommand{\fp}{\mathfrak{p}}
\newcommand{\fm}{\mathfrak{m}}
\newcommand{\fz}{\mathfrak{z}}
\newcommand{\frakg}{\mathfrak{g}}
 \newcommand{\fk}{\mathfrak{k}}

\newcommand{\fq}{\mathfrak{q}}
\newcommand{\fn}{\mathfrak{n}}
\newcommand{\fj}{\mathfrak{j}}
\newcommand{\fs}{\mathfrak{s}}
\newcommand{\ft}{\mathfrak{t}}
\newcommand{\la}{\l_\a}

\newcommand{\fUg}{U(\frakg_\C)}
\newcommand{\fUs}{U(\fs_\C)}
\newcommand{\fUa}{U(\fa_\C)}
\newcommand{\fUk}{U(\fk_\C)}
\newcommand{\fUq}{U(\fq_\C)}

\newcommand\SO{\mathop{\rm{SO}}}

\newcommand\SU{\mathop{\rm{SU}}}
\newcommand\U{\mathop{\rm{U}}}
\newcommand\Sg{\mathop{\rm{S}}}
\newcommand\Sp{\mathop{\rm{Sp}}}
\newcommand\Spin{\mathop{\rm{Spin}}}

\newcommand{\frakacs}{\mathfrak a_{\C}^*}

\newcommand{\polya}{{\rm S}(\frak a_\C)}

\newcommand{\rml}{{\rm l}}
\newcommand{\rmm}{{\rm m}}
\newcommand{\rms}{{\rm s}}

\newcommand{\hyper}[4]{\ensuremath{\sideset{_{_2}}{_{_1}}{\mathop{F}}
\left(#1,#2;#3;#4\right)}}

\newcommand{\Hreg}{A_\C^{\rm reg}}
\newcommand{\Areg}{A^{\rm reg}}

\newcommand{\wt}{\widetilde}

\renewcommand{\phi}{\varphi}

\begin{document}

\makeatletter
\title[Hypergeometric functions of type $BC$]{Differential operators, radial parts and\\ a one-parameter family of hypergeometric functions of type $BC$}
\author{E. K. Narayanan}
\address{Department of Mathematics, Indian Institute of Science, Bangalore -12, India}
\email{naru@math.iisc.ernet.in}
\author{A. Pasquale}
\address{Institut Elie Cartan de Lorraine (UMR CNRS 7502),
Universit\'e de Lorraine, F-57045 Metz, France.}
\email{angela.pasquale@univ-lorraine.fr}

\thanks{The authors gratefully acknowledge financial support by the
Indo-French Centre for the Promotion of Advanced Research (IFCPAR/CEFIPRA),
Project No. 5001--1: ``Hypergeometric functions: harmonic analysis and representation theory''.}

\date{}

\begin{abstract}
We introduce the symmetric (respectively, non-symmetric) $\tau_{-\ell}-$hyper\-geo\-metric functions associated with a root system of type $BC$ as joint eigenfunctions of a commutative algebra of differential (respectively, differential-reflection) operators. Under certain conditions on the real parameter $\ell$, we derive their properties (positivity, estimates, asymptotics and boundedness) by establishing the analogous properties for the Heckman-Opdam (symmetric and non-symmetric) hypergeometric functions corresponding to (not necessarily positive) multiplicity functions which are standard.
\end{abstract}

\maketitle

\tableofcontents
\section*{Introduction}

Let $X=G/K$ be a Riemannian symmetric space of the noncompact type. The $L^2$ spherical harmonic analysis on $X$, namely Harish-Chandra's inversion formula and the Helgason-Gangolli-Rosenberg's Plancherel and Paley-Wiener theorems for the spherical transform on $X$, has been extended into two different directions. The first one is the $L^2$ harmonic analysis on root systems, as
developed by Heckman, Opdam and Cherednik from the late 1980s to the middle 1990s; see \cite{HS,OpdamActa} and references therein. This extension originates from the fact that, by restriction to a maximally flat subspace of $X$, the
$K$-invariant functions on $X$ become Weyl group invariant functions on the
Cartan subspace $\mathfrak{a}$ corresponding to the flat subspace.
The analogue of Harish-Chandra's spherical functions are the hypergeometric functions associated with root systems, also known as Heckman-Opdam's hypergeometric functions. The second direction is the $L^2$ harmonic analysis on
homogeneous line bundles on $X$ obtained by Shimeno \cite{ShimenoEigenfunctions,ShimenoPlancherel} at the beginning of the 1990s. Shimeno's work gives something new only when $K$ has a nontrivial center, which leads to the assumption that the space $X$ is Hermitian symmetric.  Some significant steps to extend Shimeno's $L^2$ spherical analysis  to the context of special functions associated with root systems were made by Heckman in \cite[Chapter 5]{HS}.

The present paper deals with some aspects of the $L^1$ spherical harmonic analysis on homogeneous line bundles on $X$ and their generalization to root systems.
It finds its foundations in Heckman's work \cite[Chapter 5]{HS} together with the recent results of \cite{NPP} on the $L^1$ harmonic analysis on root systems.

In the symmetric space case (corresponding to the trivial line bundle), the spherical transform of a $K$-invariant function on $X$ is obtained by integrating against spherical functions. Hence, the spherical functions which are bounded can be identified with the elements of the spectrum of the (commutative) convolution algebra of $L^1$ $K$-invariant functions on $X$. They have been determined by Helgason and Johnson in \cite{HeJo}. Let $\mathfrak a^*$ and $\mathfrak{a}_\C^*$ respectively denote the real and the complex dual of a fixed Cartan subspace $\mathfrak a$.
Consider Harish-Chandra's parametrization of the spherical functions by the elements $\mathfrak{a}_\C^*$ and let $\varphi_\lambda$ denote the spherical function corresponding to $\l\in \mathfrak{a}_\C^*$.
Let $\rho \in \mathfrak a^*$ be defined by (\ref{eq:rho}), and let $C(\rho)$ be the convex hull of the finite set $\{w\rho:w\in W\}$, where $W$ is the Weyl group.
Then the theorem of Helgason and Johnson states that the spherical function $\varphi_\l$ on $X$ is bounded if and only if $\l$ belongs to the tube domain in $\mathfrak a_\C^*$ over $C(\rho)$. It is a fundamental result in the $L^1$ spherical harmonic analysis on $X$. For instance, it implies that the spherical transform of a $K$-invariant $L^1$ function on $X$ is holomorphic in the interior
of the tube domain over $C(\rho)$. In particular, the spherical transform of a $L^1$ function cannot have compact support, unlike what happens for the classical Fourier transform.

The theorem of Helgason and Johnson has been recently extended to Heckman-Opdam's hypergeometric functions in \cite{NPP}. In the context of special functions associated with root systems, the symmetric space $X$ is replaced by a triple $(\mathfrak{a},\Sigma,m)$ where
$\mathfrak{a}$ is a Euclidean real vector space (playing the role of the Cartan subspace of $X$), $\Sigma$ is a root system in the real dual $\mathfrak{a}^*$ of $\mathfrak{a}$, and $m:\Sigma\to \C$ is a multiplicity function. As in the geometric situation, the hypergeometric functions $F_\lambda$ on $(\mathfrak{a},\Sigma,m)$ are parametrized by $\mathfrak{a}_\C^*$,  and the same formula used in the geometric case defines a certain element $\rho \in \mathfrak a^*$. As shown in \cite{NPP}, the hypergeometric function $F_\l$ is bounded if and only if $\l$ belongs to the tube domain in $\mathfrak a_\C^*$ over $C(\rho)$.
Despite the perfect analogy of these statements, the proofs for the hypergeometric
case involve new tools, due to the lack of integral formulas for the Heckman-Opdam's hypergeometric functions.

For homogeneous line bundles over a Hermitian symmetric space $X=G/K$ and their generalizations over root systems, the $L^1$ spherical (or hypergeometric) harmonic analysis has not been studied. In particular, in the literature there is no characterization of the corresponding spherical or hypergeometric functions which are bounded. In this paper, we provide such a characterization under certain assumptions which are related to the notion of standard multiplicities from \cite{HS}.

To state more precisely the results of the paper, we have to fix some notation. (We
refer to subsection \ref{subsection:hermitian} for more details.) Suppose for simplicity that $G \subset G_\C$, where the complexification $G_\C$ is connected and simply connected, and that $K$ is maximal compact in $G$. The one dimensional unitary representations of $K$ are parametrized by $\mathbb{Z}$.
For $\ell \in \mathbb{Z}$, let $E_\ell$ denote the homogenous line bundle on $G/K$ associated with the one dimensional representation $\tau_\ell$ of $K$. The smooth $\tau_{-\ell}-$spherical sections of $E_\ell$ can be identified with the space $C_\ell^\infty(G)$ of functions on $G$ satisfying
$f(k_1gk_2)=\tau_{-\ell}(k_1k_2)f(g)$ for all $g\in G$ and $k_1,k_2\in K$. Let
$\D_\ell(G/K)$ denote the (commutative) algebra of the $G$-invariant differential operators on $E_\ell$. Then the $\tau_{-\ell}$-spherical functions on $G/K$ are the (suitably normalized) joint eigenfunctions of $\D_\ell(G/K)$ belonging to $C_\ell^\infty(G)$.  For a fixed $\ell$, they are parametrized by $\lambda\in \mathfrak a_\C^*$ (modulo the Weyl group). We denote by $\varphi_{\ell,\lambda}$ the $\tau_{-\ell}-$spherical function of spectral parameter $\lambda$.
If we suppose that the trivial representation of $K$ corresponds to $\ell=0$, then $C_0^\infty(G)$ is the space of the smooth $K$-invariant functions on $G/K$, $\D_0(G/K)$ coincides with the algebra $\D(G/K)$ of $G$-invariant differential operators on $G/K$, and the $\tau_{0}$-spherical functions $\varphi_{0,\lambda}$ are precisely Harish-Chandra's spherical functions $\varphi_{\lambda}$ on $G/K$.

By the decomposition $G=KAK$, a function $f\in C_\ell^\infty(G)$ is uniquely determined by its restriction to the Cartan supspace $\mathfrak{a}\equiv A$, and this restriction is Weyl group invariant. As proved by Shimeno \cite{ShimenoEigenfunctions, ShimenoPlancherel} and Heckman \cite[Chapter 5]{HS}, the restriction to $\mathfrak{a}$ of the $\tau_{-\ell}$-spherical function
$\varphi_{\ell,\lambda}$ coincides  -- up to multiplication by the power $u^{-\ell}$ of a certain function $u$ which only depends on the root structure of $G/K$, see \eqref{eq:u} -- with the Heckman-Opdam hypergeometric function $F_\lambda$ for a multiplicity function $m(\ell)$ which is a deformation of the multiplicity function $m$ of the symmetric space $G/K$; see \eqref{eq:mult-l}. Recall that the root system $\Sigma$ of $G/K$ is of type $BC_r$
and that the multiplicity $m$ is a triple of nonnegative integers $(m_\rms,m_\rmm,m_\rml)$ with $m_\rml=1$.
This leads to the definition of the functions
\begin{equation}
\label{Fell-intro}
F_{\ell,\lambda}(m,x)=u(x)^{-\ell} F_\lambda(m(\ell),x)\,, \qquad (x\in \mathfrak{a})
\end{equation}
for $\ell\in \C$ and arbitrary triples $(\mathfrak{a},\Sigma,m)$ with $\Sigma$ of type $BC_r$ and $m_\rml=1$. These functions have already been considered in \cite[Chapter 5]{HS} as well as in other related papers, as in \cite{HoOl14} and \cite{Oda14}. Following the terminology of the geometric case, we will call them \textit{$\tau_{-\ell}-$hypergeometric functions}.  Likewise, one can introduce
\textit{non-symmetric $\tau_{-\ell}-$hypergeometric functions} $G_{\ell,\lambda}(m,x)$, linked to Opdam's nonsymmetric hypergeometric functions $G_\lambda$
by the relation
\begin{equation}
\label{Gell-intro}
G_{\ell,\lambda}(m,x)=u(x)^{-\ell} G_\lambda(m(\ell),x)\,, \qquad (x\in \mathfrak{a})\,.
\end{equation}

The equation \eqref{Fell-intro} naturally suggests to deduce the properties
of the $\tau_{-\ell}-$hypergeometric functions (and hence of the  $\tau_{-\ell}-$spherical functions in the geometric case), out of those of the
hypergeometric functions $F_\lambda$.
However, even if $m$ is nonnegative (i.e. $m_a\geq 0$ for
$a\in \{\rms,\rmm,\rml\}$), the deformed multiplicity function $m(\ell)$ is not. In the literature, all the detailed properties of the Heckman-Opdam hypergeometric functions $F_\lambda(m,x)$ (such as positivity results, asymptotics, sharp estimates, boundedness) are known only under the assumption that $m$ is nonnegative. So they cannot be applied to the present situation. To our knowledge, the only exceptions are the estimates determined in \cite{HoOl14}. But also for using them, we will need to further extend the domain of multiplicities where they hold.

Let $\mathcal{M}$ be the set of all real-valued multiplicity functions on the fixed root system $\Sigma$, and set
$$
\mathcal{M}_1=\{(m_\rms,m_\rmm,m_\rml)\in \mathcal{M}:  m_\rmm> 0, m_\rms> 0, m_\rms+2m_\rml> 0\}\,.
$$
The elements of $\mathcal{M}_1$ were called in \cite[Definition 5.5.1]{HS}
the \textit{standard multiplicities}. They contain the positive multiplicity functions and their definition is motivated by the regularity of Harish-Chandra's $c$-function; see \cite[Lemma 5.5.2]{HS} and \eqref{eq:b-nonzero} below. Observe that $\rho(m)\neq 0$ for $m\in \mathcal{M}_1$, but $\rho(m)=0$ for nonzero multiplicities $m$ on the boundary of $\mathcal{M}_1$.

Consider the Heckman-Opdam's (symmetric and non-symmetric) hypergeometric functions $F_\lambda(m,x)$ and $G_\lambda(m,x)$; see \ref{subsection:hyp} for the definitions.
Under the condition that $m\in \mathcal{M}_1$ we prove:
\begin{enumerate}
\item $F_\lambda$ and $G_\lambda$ are real and strictly positive on $\mathfrak{a}$ for all $\l\in \mathfrak{a}^*$,
\item $|F_\lambda|\leq F_{\Re\lambda}$ and $|G_\lambda|\leq G_{\Re\lambda}$
on $\mathfrak{a}$ for all $\l\in \mathfrak{a}^*_\C$,
\item $\max\{|F_\lambda(x)|, |G_\lambda(x)|\}\leq \sqrt{|W|} e^{\max_{w\in W}(w\l)(x)}$ for all $\lambda \in \mathfrak{a}^*_\C$ and $x\in
\mathfrak{a}$,
\item $F_{\lambda+\mu}(x)\leq F_{\mu}(x) e^{\max_{w\in W}(w\l)(x)}$
and $G_{\lambda+\mu}(x)\leq G_{\mu}(x) e^{\max_{w\in W}(w\l)(x)}$ for all
$\lambda \in \mathfrak{a}^*$, $\mu \in \overline{(\mathfrak{a}^*)^+}$ and $x\in
\mathfrak{a}$,
\item asymptotics on $\mathfrak{a}$ for $F_\lambda$ when $\lambda\in  \mathfrak{a}^*_\C$ is fixed but not necessarily regular,
\item sharp estimates on $\mathfrak{a}$  for $F_\lambda$ when $\lambda \in \mathfrak{a}^*$ is fixed but not necessarily regular.
\end{enumerate}
The estimates (but not the asymptotics) pass by continuity to the boundary of
$\mathcal{M}_1$ as well. The above properties extend to $\mathcal{M}_1$ the corresponding properties proved for nonnegative multiplicities by Opdam \cite{OpdamActa}, Schapira \cite{Sch08}, R\"osler, Koornwinder and Voit  \cite{RKV13}, and the authors and Pusti \cite{NPP}. (The estimates (3) hold
in fact for a larger set of multiplicities than $\mathcal{M}_1$; see Lemma \ref{lemma:OpdamM2M3}.)

To deduce from the above and \eqref{Fell-intro}-\eqref{Gell-intro}, the corresponding properties of the $\tau_{-\ell}-$hyper\-geo\-metric functions $F_{\ell,\lambda}$ and $G_{\ell,\lambda}$, we consider the deformed multiplicity functions $m(\ell)$, where $m$ is nonnegative, as special instances of standard multiplicities. This is possible provided the range of $\ell$ is limited to $\ell\in ]\ell_{\rm min},\ell_{\rm max}[$, where $\ell_{\rm min}=-\frac{m_\rms}{2}$ and $\ell_{\rm max}=\frac{m_\rms}{2}+1$. Notice that,
since $F_{-\ell,\lambda}=F_{\ell,\lambda}$, see \eqref{eq:Fellm-even-ell}, the properties of $F_{\ell,\lambda}$ with $\ell\in ]\ell_{\rm min},\ell_{\rm max}[$   automatically extend to the larger domain $|\ell|<\ell_{\rm max}$.

The resulting properties of the $\tau_{-\ell}-$hypergeometric functions allow us to characterize those that are bounded, provided $|\ell|<\ell_{\rm max}$:
exactly as in the case $\ell=0$, the function $F_{\ell,\lambda}$ is bounded if and only if $\lambda$ is in the tube in $\mathfrak{a}_\C^*$ above $C(\rho(m))$; see Theorem \ref{thm:bdd}. This in particular yields, under the same assumptions on $\ell$, the characterization of the bounded $\tau_{-\ell}-$spherical functions.
One has to point out that, in the geometric case, one of the two directions of this
characterization (without any restriction on $\ell$) is an immediate consequence of the integral formula \eqref{eq:integral-formula} for $\varphi_{\ell,\lambda}$ together with the corresponding property for Harish-Chandra's spherical functions; see Corollary \ref{Cor:estimates-spherical}. On the other hand, unless $\ell=0$, the integral formula does not show our basic result that $\varphi_{\ell,\lambda}$ is positive for $|\ell|<\ell_{\rm max}$.

Our paper is organized as follows. In section 1 we introduce the notation and collect some preliminary results on line bundles over Hermitian symmetric spaces and on the Heckman-Opdam's theory of hypergeometric functions. In particular, we recall the definition of the Heckman-Opdam hypergeometric functions as joint eigenfunctions of a certain commutative algebra $\D(m)$ of finite rank. The algebra $\D(m)$ is explicitly constructed inside the algebra $\D^W_\mathcal{R}$ of Weyl group invariant differential operators on $A\equiv \mathfrak{a}$ with coefficients in a specific algebra $\mathcal{R}$ of functions; see \eqref{eq:galpha} and Lemma \ref{lemma:commutatorL}.

In section 2 we restrict ourselves to the geometric case and consider the $\tau_{-\ell}-$radial components $\delta_\ell(D)$ of the differential operators $D\in \D_\ell(G/K)$.  For the Casimir operator $\Omega$, the
$\tau_{-\ell}-$radial component $\delta_\ell(\Omega)$ has been computed (with different methods) by Shimeno \cite[Lemma 2.4]{ShimenoPlancherel} and by Heckman \cite[Theorem 5.1.7]{HS}. The nature of the $\tau_{-\ell}-$radial components of the other elements of $\D_\ell(G/K)$ can be deduced from the work of Harish-Chandra  \cite{HC60}: they are Weyl group invariant differential operators on $A\equiv \mathfrak{a}$ with coefficients in a certain algebra $\mathcal{R}_0$, see \eqref{eq:R0}. The case of the trivial representation (i.e. when $\ell=0$) is special. Indeed, as proved by Harish-Chandra \cite[sections 5 and 6]{HC58}, the radial component $\delta_0(D)$ of any differential operator $D\in\D(G/K)$ belongs to $\D^W_\mathcal{R}$. Here $\mathcal{R}$ is the same algebra appearing in Heckman-Opdam's theory. In fact, in the geometric case, the algebra $\D(m)\subset \D^W_\mathcal{R}$ constructed by Heckman and Opdam is precisely $\delta_0(\D(G/K))$.
The algebra $\mathcal{R}_0$ for the case $\ell\neq 0$ is bigger than $\mathcal{R}$. In \cite[Theorem 5.1.4]{HS}, Heckman states the crucial fact that the smaller algebra $\mathcal{R}$ suffices for all $\ell\in \Z$. One can see it
for $\delta_\ell(\Omega)$ directly from its explicit formula, for the other differential operators one needs more work to prove this fact. We do it in Theorem
\ref{thm:coefficients-taul-radial}, which is the main result of section 2. Its proof uses Harish-Chandra's construction, which we outline for the case of the one-dimensional $K$-type $\tau_{-\ell}$ in Proposition \ref{prop:radial1}, and the vanishing of $\tau_{-\ell}$ on certain elements appearing in this construction, see Lemma \ref{lemma:vanishing-tauell}.

The fact that $\mathcal{R}$ suffices for all $\ell\in \Z$ is crucial because one needs it to identify $\delta_\ell(\D_\ell(G/K))$ inside
$\D^W_\mathcal{R}$ as $u^{-\ell} \circ \D(m(l)) \circ u^{\ell}$, where
$\D(m(l))$ is as in the usual Heckman-Opdam's theory. Having this, even in the case of non-geometric triples $(\mathfrak{a},\Sigma,m)$ and $\ell\in \C$, one can define the $\tau_{-\ell}-$hypergeometric functions as (Weyl group invariant, regular and suitably normalized) joint eigenfunctions of the commutative algebra $\D_\ell(m):=u^{-\ell} \circ \D(m(l)) \circ u^{\ell}$,
and they will coincide with restriction to $\mathfrak{a}$ of the $\tau_{-\ell}-$spherical functions in the geometric case.

The (symmetric and nonsymmetric) $\tau_{-\ell}-$hypergeometric functions $F_{\ell,\lambda}$ and $G_{\ell,\lambda}$ are defined in section 3. For the nonsymmetric ones, we introduce Cherednik-like
operators and we compute the $\tau_{-\ell}-$Heckman-Opdam Laplacian.
All definitions and properties in the first part of the section are a straightforward consequence of the fact that the algebra $\D_\ell(m)$ is obtained by conjugation
by $u^{-\ell}$ of the algebra $\D(m(l))$, for which all these objects are defined.
In particular, we recover \eqref{Fell-intro} and \eqref{Gell-intro}. In the last part of this section, we look more closely at the two main examples: the geometric case and the rank-one case.

The analysis of  $\tau_{-\ell}-$hypergeometric functions starts in section 4. First we study the basic properties (positivity, estimates and asymptotics) of the Heckman-Opdam's (symmetric and non-symmetric) hypergeometric functions $F_{\lambda}$ and $G_{\lambda}$ for multiplicity functions $m$ that need not be nonnegative. We single out several sets of real-valued multiplicities where the above properties hold. All these properties turn out to be satisfied on the set $\mathcal{M}_1$ of standard multiplicities. In the later part of the section, we deduce the corresponding properties for the $\tau_{-\ell}-$hypergeometric functions $F_{\ell,\lambda}$ and $G_{\ell,\lambda}$ for the values of $\ell\in \R$ for which the deformed multiplicity $m(\ell)$ is standard. In the final section we use the results of section 4 to prove the characterization of the bounded $\tau_{-\ell}-$hypergeometric functions $F_{\ell,\lambda}$ under the assumption that $|\ell|<\ell_{\rm max}$ (Theorem \ref{thm:bdd}).

\section{Notation and preliminaries} \label{section:notation}

We shall use the standard notation $\N$, $\N_0$, $\Z$, $\R$,
$\C$ and $\C^\times$ for the positive integers, the nonnegative integers, the
integers, the reals, the complex numbers and the nonzero
complex numbers. The symbol $A\cup B$
denotes the union of $A$ and $B$, whereas $A \sqcup B$ indicates
their disjoint union.
If $F(m)$ is a function on a space $X$ that depends on a parameter $m$,
we will denote its value at $x\in X$ by $F(m;x)$ rather than $F(m)(x)$.

Given two nonnegative functions $f$ and $g$
on a same domain $D$, we write $f \asymp g$ if there exists
positive constants $C_1$ and $C_2$ so that $C_1 \leq
\frac{f(x)}{g(x)} \leq C_2$ for all $x \in D$.

\subsection{Root systems}
\label{subsection:roots}

Let $\frak a$ be an $r$-dimensional real Euclidean vector space
with inner product $\inner{\cdot}{\cdot}$, and let $\frak a^*$ be
the dual space of $\frak a$. For $\l\in \frak a^*$ let $x_\l \in
\mathfrak a$ be determined by the condition that
$\l(x)=\inner{x}{x_\l}$ for all $x \in \mathfrak a$. The assignment
$\inner{\l}{\mu}:=\inner{x_\l}{x_\mu}$ defines an inner product in
$\frak a^*$. Let $\frak a_\C$ and $\frak a_\C^*$ respectively
denote the complexifications of $\frak a$ and $\frak a^*$. The
$\C$-bilinear extension to $\frak a_\C$ and $\frak a_\C^*$ of the
inner products on $\frak a^*$ and $\frak a$ will also be indicated
by $\inner{\cdot}{\cdot}$.
We denote by  $|\cdot|=\inner{\cdot}{\cdot}^{1/2}$ the associated norm.
We shall often employ the notation
\begin{equation}\label{eq:la}
  \la:=\frac{\inner{\l}{\a}}{\inner{\a}{\a}}
\end{equation}
for elements $\l,\a\in \fa_\C^*$ with $|\a|\neq 0$.

Let $\Sigma$ be a root system in $\frak a^*$
with set of positive roots of the form
$\Sigma^+=\Sigma^+_\rms \sqcup \Sigma^+_\rmm \sqcup \Sigma^+_\rml$, where
\begin{equation}
\label{eq:roots}
\Sigma^+_\rms=\big\{\frac{\beta_j}{2}: 1\leq j \leq r\big\}\,,  \notag\, \quad
\Sigma^+_\rmm=\big\{\frac{\beta_j\pm \beta_i}{2}: 1\leq i < j \leq r\big\}\,,\quad
\Sigma^+_\rml=\{\beta_j: 1\leq j \leq r\} \,.
\end{equation}
We assume that the elements of $\Sigma^+_\rml$ form an orthogonal basis of $\fa^*$ and that they all have the same norm $p$.
We denote by $W$ the Weyl group of $\Sigma$. It is the finite group of orthogonal transformations in $\fa^*$  generated by the reflections $r_\a$, with $\a \in \Sigma$,  defined by $r_\a(\l):=\l-2 \la \a$ for all $\l \in \mathfrak a^*$.  It acts on the roots by permutations and sign changes. For  $a\in \{\rms,\rmm,\rml\}$ set $\Sigma_a=\Sigma^+_a \sqcup (-\Sigma^+_a)$.  Then each $\Sigma_a$ is a Weyl group orbit in $\Sigma$.

A multiplicity function on $\Sigma$ is a $W$-invariant function on $\Sigma$. It is therefore given by a triple $m=(m_\rms,m_\rmm,m_\rml)$ of complex numbers so that $m_{a}$ is the (constant) value of $m$ on $\Sigma_{a}$ for ${a}\in \{\rms,\rmm,\rml\}$.
We suppose in the following that $m_\rml=1$.

It will be convenient to refer to a root system $\Sigma$ as above as of type $BC_r$ even if some values of $m$ are zero. This means that root systems of type $C_r$ will be considered as being of type $BC_r$ with $m_\rms=0$ and $m_\rmm\neq 0$. Likewise, the direct products of rank-one root systems $(BC_1)^r$ and $(A_1)^r$ will be considered of type $BC_r$ with $m_\rmm=0$ and with $m_\rmm=m_\rms=0$, respectively.

The dimension $r$ of $\frak a$ is called the \emph{(real) rank} of
the triple $(\frak a, \Sigma, m)$. Finally, we set
\begin{equation}
\label{eq:rho}
 \rho(m)=\frac{1}{2} \sum_{\a \in \Sigma^+} m_\a \a=\frac{1}{2}  \sum_{j=1}^r \Big(\frac{m_\rms}{2}+1+(j-1)m_\rmm\Big) \beta_j \in \fa^*\,.
\end{equation}

\begin{Rem}
\label{rem:hermitian}
For special values of the multiplicities $m=(m_\rms,m_\rmm,m_\rml)$, triples $(\frak a, \Sigma, m)$ as above appear as restricted root systems of irreducible Hermitian symmetric spaces $G/K$ of the noncompact type.  See \cite{He1} or \cite{ShimenoPlancherel}. See also subsection \ref{subsection:hermitian} below.
\medskip

\begin{table}[h]
\setlength{\extrarowheight}{2pt}
\caption{Geometric multiplicities}
\begin{tabular}{|c|c|c|c|}
\hline
$G$ &$K$ & $\Sigma$ &$(m_\rms,m_\rmm,m_\rml)$ \\[.2em]
\hline\hline
$\SU(p,q)$ & $\Sg(\U(p)\times \U(q))$ &
\begin{tabular}{l}$p=q$: $C_p$ \\$p<q$: $BC_p$ \end{tabular}
&\begin{tabular}{l} $p=q$: $(0,2,1)$\\ $p<q$: $(2(q-p),2,1)$\end{tabular}  \\[.2em]
\hline
$\SO_0(p,2)$ &$\SO(p)\times\SO(2)$ & $C_2$ &$(0,p-2,1)$ \\[.2em]
\hline
$\SO^*(2n)$ & $\U(n)$ & \begin{tabular}{l} $n$ even: $C_n$\\
$n$ odd: $BC_n$ \end{tabular}  & \begin{tabular}{l} $n$ even: $(0,4,1)$\\ $n$ odd: $(4,4,1)$\end{tabular} \\[.2em]
\hline
$\Sp(n,\R)$ &$\U(n)$ & $C_n$ &$(0,1,1)$ \\[.2em]
\hline
$\mathfrak{e}_{6(-14)}$ &$\Spin(10)\times\U(1)$ & $BC_2$ &$(8,6,1)$ \\[.2em]
\hline
$\mathfrak{e}_{7(-25)}$ &$\ad(\mathfrak{e}_6)\times \U(1)$ & $C_3$ &$(0,8,1)$ \\[.2em]
\hline
\end{tabular}
\end{table}

The literature on Hermitian symmetric spaces refers to the situations where $\Sigma=C_r$ and
$\Sigma=BC_r$ as to the \textit{Case I} and the \textit{Case II}, respectively.
The values $m=(m_\rml,m_\rmm,m_\rms)$ appearing in Table 1  will be called \textit{geometric multiplicities}.
\end{Rem}

Notice that in this paper we adopt the notation commonly used in the theory of symmetric spaces. It differs from the notation in the work of Heckman and Opdam in the following ways. The root system $R$ and the multiplicity function $k$ used by Heckman and Opdam
are related to our $\Sigma$ and $m$ by the relations $R=\{2\a:\a \in \Sigma\}$ and $k_{2\a}=m_\a/2$ for $\a \in \Sigma$.

The complexification $\frak a_\C$ of $\frak a$ can be viewed as
the Lie algebra of the complex torus $A_\C:= \frak a_\C / \Z\{i\pi
x_\a: \a \in \Sigma\}$. We write $\exp: \frak a_\C \rightarrow
A_\C$ for the exponential map, with multi-valued inverse $\log$.
The split real form $A:=\exp \frak a $ of $A_\C$ is an abelian
subgroup with Lie algebra $\frak a$ such that $\exp: \frak a
\rightarrow A$ is a diffeomorphism. In the following, to simplify
the notation, we shall identify $A$ with $\mathfrak a$ by means of
this diffeomorphism.

The action of $W$ extends to $\frak a$ by duality,
and to $\frakacs$ and $\frak a_\C$ by $\C$-linearity.
For $\lambda \in \fa_\C^*$ we set $W_\lambda=\{w\in W:w\lambda=\lambda\}$.
Furthermore, the action of $W$ on a space $X$ extends to an action on functions $f$ on $X$
by $(wf)(x):=f(w^{-1}x)$, $w \in W$, $x\in X$.

The positive Weyl chamber $\frak a^+$ consists of the elements  $x\in\frak a$ for which
$\a(x)>0$ for all $\a \in \Sigma^+$; its closure is
$\overline{\frak a^+}=\{x \in \frak a: \text{$\a(x) \geq 0$ for all $\a \in \Sigma^+$}\}$.
Dually, the positive Weyl chamber $(\frak a^*)^+$ consists of the elements  $\l\in\frak a^*$ for which
$\inner{\l}{\a}>0$ for all $\a \in \Sigma^+$. Its closure is denoted $\overline{(\frak a^*)^+}$.
The sets $\overline{\frak a^+}$ and $\overline{(\frak a^*)^+}$
are fundamental domains for the action of $W$ on $\frak a$ and $\frak a^*$, respectively.

The restricted weight lattice of $\Sigma$ is
  $P=\{\l \in \frak a^*:\la\in\Z \quad \text{for all $\a \in\Sigma$}\}.$
Observe that $\{2\a:\a\in \Sigma\} \subset P$.
If $\l \in P$, then the exponential $e^\l:A_\C\rightarrow \C$ given
by $e^\l(h):=e^{\l(\log h)}$ is single valued. The $e^\l$ are the algebraic characters
of $A_\C$. Their $\C$-linear span coincides with the ring
of regular functions $\C[A_\C]$ on the affine algebraic variety $A_\C$.
The lattice $P$ is $W$-invariant, and the Weyl group acts on $\C[A_\C]$
according to $w(e^\l):=e^{w\l}$.
The set $\Hreg:=\{h \in A_\C: e^{2\a(\log h)}\neq 1 \ \text{for all $\a \in \Sigma$}\}$
consists of the regular points of $A_\C$ for the action of $W$. Notice that $A^+ \equiv \mathfrak a^+$ is a subset of $\Hreg$.  The algebra $\C[\Hreg]$ of regular functions on $\Hreg$
is the subalgebra of the quotient field of $\C[A_\C]$ generated by
$\C[A_\C]$ and by $1/(1-e^{-2\a})$ for $\a \in \Sigma^+$.

\subsection{Cherednik operators and hypergeometric functions}
\label{subsection:hyp}
In this subsection, we review some basic notions on the hypergeometric functions associated with root systems. This theory has been developed by Heckman, Opdam and Cherednik. We refer the reader to \cite{HS}, \cite{OpdamActa}, and \cite{Opd00} for more details and further references.

Let $\polya$ denote the symmetric algebra over $\frak a_\C$ considered as the
space of polynomial functions on $\frakacs$, and let $\polya^W$ be the
subalgebra of $W$-invariant elements.
Every $p \in \polya$ defines a
constant-coefficient differential operators $p(\partial)$ on
$A_\C$ and on $\frak a_\C$
such that $\xi(\partial)=\partial_\xi$ is the directional derivative in the
direction of $\xi$ for all $\xi \in \frak a$. The algebra of the differential operators $p(\partial)$
with $p \in \polya$ will also be indicated by $\polya$.

Let $\mathcal{R}$ denote the subalgebra of $\C[\Hreg]$ generated by the functions $1$ and
\begin{equation}
\label{eq:galpha}
g_\a=(1-e^{-2\a})^{-1} \qquad (\a\in \Sigma^+)\,.
\end{equation}
Notice that for $\xi\in \fa$ and $\a\in \Sigma^+$ we have
\begin{eqnarray}
\partial_\xi g_\a&=&2\a(\xi)  \big(g_\a^2-g_\a\big) \notag\\
\label{eq:fminusa}
g_{-\a}&=&1-g_\a\,.
\end{eqnarray}
Hence $\mathcal{R}$ is stable under the actions of  $S(\fa_\C)$ and of the Weyl group.
Let $\D_\mathcal{R}=\mathcal{R}\otimes S(\fa_\C)$ be the algebra of differential operators on $\Hreg$ with coefficients in $\mathcal{R}$. The Weyl group $W$ acts on $\D_\mathcal{R}$
according to
\begin{equation*}
w\big(\phi\otimes p(\partial)):=w\phi \otimes (wp)(\partial).
\end{equation*}
We indicate by $\D_\mathcal{R}^W$ the subalgebra of $W$-invariant elements of $\D_\mathcal{R}$.
The space
$\D_\mathcal{R} \otimes \C[W]$ can be endowed with the structure of an associative algebra
with respect to the product
\begin{equation*}
 (D_1 \otimes w_1)\cdot (D_2 \otimes w_2)=D_1w_1(D_2) \otimes w_1w_2,
\end{equation*}
where the action of $W$ on differential operators is defined by
$(wD)(wf):=w(Df)$ for every sufficiently differentiable function
$f$.
Considering $D \in \D_\mathcal{R}$ as an element of $\D_\mathcal{R} \otimes \C[W]$, we shall
usually write $D$ instead of $D \otimes 1$. The elements of the
algebra $\D_\mathcal{R} \otimes \C[W]$ are differential-reflection operators on $\Hreg$.
The differential-reflection operators act on functions $f$ on $\Hreg$
according to $(D\otimes w)f:=D(wf)$.

The differential component of a differential-reflection operator $P \in \D_{\mathcal R}\otimes \C[W]$ is the differential operator $\beta(D)\in\D_{\mathcal R}$ such that
\begin{equation}
\label{eq:beta}
Pf=\beta(P)f\,.
\end{equation}
for all $f\in \C[A_\C]^W$. If $P=\sum_{w\in W} P_w\otimes w$ with $P_w\in \D_{\mathcal R}$,
then $\beta(P)=\sum_{w\in W} P_w$.

For $\xi \in \frak a$ the Cherednik operator (or Dunkl-Cherednik operator\/)  $T_\xi(m)\in \D_\mathcal{R} \otimes \C[W]$
is defined by
\begin{equation}
\label{eq:T}
  T_\xi(m):=
\partial_\xi-\rho(m;\xi)+\sum_{\a\in \Sigma^+} m_\a \a(\xi) (1-e^{-2\a})^{-1}
\otimes (1-r_\a)
\end{equation}
(Recall the notation $\rho(m;\xi)$ for $\rho(m)(\xi)$.)

The Cherednik operators can be considered as operators acting
on smooth functions on $\frak a$ (or $A$). This is possible because, as can
be seen from the Taylor formula, the term $1-r_\a$ cancels the
apparent singularity arising from the denominator $1-e^{-2\a}$.

The Cherednik operators commute with each other (cf.
\cite{OpdamActa}, Section 2). Therefore the map  $\xi \to T_\xi(m)$
on $\frak a$ extends uniquely to an algebra homomorphism $p
\to T_p(m)$ of $\polya$ into $\D_\mathcal{R} \otimes \C[W]$.
If $p\in \polya^W$, then $D_p(m):=\beta\big(T_p(m)\big)$ turns out to be in $\D_\mathcal{R}^W$.

Let $p_L \in \polya^W$ be defined by $p_L(\l):=\inner{\l}{\l}$ for $\l \in \frakacs$.
Then $T_{p_L}(m)=\sum_{j=1}^r T_{\xi_j}(m)^2$, where  $\{\xi_j\}_{j=1}^r$ is any orthonormal basis of
$\fa$, is called the Heckman-Opdam Laplacian. Explicitly,
\begin{equation}
\label{eq:HOLaplacian}
T_{p_L}(m)=L(m)+\inner{\rho(m)}{\rho(m)}-\sum_{\a\in \Sigma^+} m_\a \frac{\inner{\a}{\a}}{\sinh^2\alpha} \otimes (1-r_\a)\,,
\end{equation}
where
\begin{equation}
    \label{eq:Laplm}
  L(m):=L_{\mathfrak a}+\sum_{\a \in \Sigma^+} m_\a \,\coth\a \;
       \partial_{x_\a}\,,
\end{equation}
$L_{\mathfrak a}$ is the Laplace operator on $\mathfrak a$, and
$\partial_{x_\a}$ is the directional derivative in the direction
of the element $x_\a\in \fa$  corresponding to $\a$ in the
identification of $\fa$ and $\fa^*$ under $\inner{\cdot}{\cdot}$,
as at the beginning of subsection \ref{subsection:roots}.  See
\cite[(1)]{Sch08} and \cite[\S 2.6]{SchThese} for the computation
of \eqref{eq:HOLaplacian}. In \eqref{eq:HOLaplacian} and
(\ref{eq:Laplm}) we have set 
\begin{eqnarray}
\label{eq:sinh-alpha}
\sinh \a&=&\frac{e^\a-e^{-\a}}{2}=\frac{e^\a}{2}\, (1-e^{-2\a})\,,\\
\label{eq:coth-alpha}
\coth \a&=&\frac{1+e^{-2\a}}{1-e^{-2\a}}=\frac{2}{1-e^{-2\a}}-1\,.
\end{eqnarray}
Moreover,
$$D_{p_L}(m) =L(m)+ \inner{\rho(m)}{\rho(m)}\,.$$

Set $\D(m):=\{D_p(m):p \in \polya^W\}$. The map $\gamma(m): \D(m)\rightarrow \polya^W$ defined by
\begin{equation}
  \label{eq:HChomo}
\gamma(m)\big(D_p(m)\big)(\l):=p(\l)
\end{equation}
is called the Harish-Chandra homomorphism.
It defines an algebra isomorphism of $\D(m)$ onto $\polya^W$ (see \cite[Theorem 1.3.12 and Remark 1.3.14]{HS}). From Chevalley's theorem it therefore follows that
$\D(m)$ is generated by $r (=\dim \frak a)$ elements. The next lemma will play a decisive role for us.

\begin{Lemma}
\label{lemma:commutatorL}
For every multiplicity function $m$, the algebra $\D(m)$ is the centralizer of $L(m)$ in $\D_\mathcal{R}^W$.
\end{Lemma}
\begin{proof}
This is \cite[Theorem 1.3.12]{HS}. See also \cite[Remark 1.3.14]{HS} for its extension to
complex-valued multiplicities.
\end{proof}

Let $\l \in \frakacs$ be fixed.
The Heckman-Opdam hypergeometric function with spectral parameter $\l$ is the
unique $W$-invariant analytic solution $F_\l(m)$ of the system of differential equations
\begin{equation}
  \label{eq:hypereq}
D_p(m) f=p(\l)f \qquad (p \in \polya^W),
\end{equation}
which satisfies $f(0)=1$.
The non-symmetric hypergeometric function with spectral parameter $\l$ is the
unique analytic solution $G_\l(m)$ of the system of differential equations
\begin{equation}
  \label{eq:hypereq}
T_\xi(m) g=\l(\xi) g \qquad (\xi \in \fa),
\end{equation}
which satisfies $g(0)=1$.
These functions are linked by the relation
\begin{equation}
\label{eq:F-G}
F_\l(m;x)=\frac{1}{|W|}\, \sum_{w\in W} G_\l(m;w^{-1}x) \qquad (x \in \fa)\,,
\end{equation}
where $|W|$ denotes the cardinality of $W$.

For geometric multiplicities, $\D(m)$ coincides with the algebra
of radial components of the $G$-invariant differential operators
on $G/K$. The fact that the radial components of $G$-invariant
differential operators on $G/K$ have coefficients in $\mathcal{R}$
was proven by Harish-Chandra. See \cite[sections 5 and 6]{HC58}.
Furthermore, $F_\l(m)$ agrees with the restriction to $A\equiv
\fa$  of Harish-Chandra's spherical function on $G/K$ with
spectral parameter $\l$. A geometric intepretation of the
functions $G_\l(m)$ has been recently given in \cite{Oda14}.

\subsection{Hermitian symmetric spaces of the noncompact type}
\label{subsection:hermitian}

In this subsection, we recall a few facts on the fine structure of the Hermitian symmetric spaces and their homogeneous line bundles. For further details, we refer to \cite[Ch. VIII]{He1}, \cite{KW65}, \cite{Schl84},  \cite{ShimenoEigenfunctions} and \cite{ShimenoPlancherel}.

\subsubsection{Structure theory}
Let $G/K$ denote an irreducible Hermitian symmetric space of the noncompact type, as in the table in Remark \ref{rem:hermitian}. By possibly replacing $G$ by its universal covering, we can suppose that $G$ is simply connected. Then $K$ is the direct product of a compact semisimple group $K_s$ and a one dimensional vector group $K_a$. Let $\frakg$, $\fk, \fk_s$ and $\fz$ be the Lie algebras of $G$, $K$, $K_s$ and $K_a$, respectively. Then  $\fk_s=[\fk,\fk]$, $\dim\fz=1$ and $\fk=\fk_s+\fz$. We denote by $\frakg=\fk\oplus \fp$ the Cartan decomposition and $\theta$ the corresponding Cartan involution. The complexification of $\frak g$ is denoted by $\frakg_\C$. If $\fs\subset \frakg$ is a subalgebra, we denote by $\fs_\C$ the complexification of $\fs$ and we suppose that $\fs_\C\subset \frakg_\C$. The real and complex duals of $\fs$ are indicated by $\fs^*$ and $\fs_\C^*$, respectively.

Let $\ft$ be a Cartan subalgebra of $\fk$. Then
$\fz\subset \ft$ and $\ft$ is a Cartan subalgebra of $\frak g$ as well. The set of the roots of
$(\frakg,\ft)$ is the set $\Delta$ of all $\gamma\in i\ft^*\setminus \{0\}$ so that $$\frakg_\gamma=\{X\in \frakg:\text{$[H,X]=\gamma(H)X$ for all $H\in \fa$}\}$$
contains nonzero elements.
Let
$$\Delta_c=\{\gamma\in \Delta:\frakg_\gamma\subset\fk_\C\}
\quad \text{and} \quad \Delta_n=\{\gamma\in
\Delta:\frakg_\gamma\subset\fp_\C\}$$ be respectively the sets of
compact and noncompact roots. Hence
$\Delta=\Delta_c\sqcup\Delta_n$. There is an element $Z_0\in \fz$
such that $\ad Z_0$ has eigenvalues $-i,0,i$. We fix a set
$\Delta^+$ of positive roots in $\Delta$ so that
$\Delta_n^+=\{\a\in \Delta: \a(Z_0)=i\}$, where
$\Delta_n^+=\Delta^+\cap \Delta_n$.

Let $B$ denote the Killing form of $\frakg_\C$. For $\gamma \in \ft_\C^*$ let $H_\gamma \in \ft_\C$ be determined by the condition that $\gamma(H)=B(H_\gamma,H)$ for all $H\in \ft$.
Let $\{\gamma_1,\dots,\gamma_r\}\subset \Delta_n^+$ a maximal strongly orthogonal subset, chosen so that $\gamma_j$ is the highest element of $\Delta_n$ strongly orthogonal to
$\{\gamma_{j+1},\dots,\gamma_r\}$ for all $j=1,\dots,r$.
Set
$$\ft^-=\oplus_{j=1}^r \R iH_{\gamma_j} \quad \text{and} \quad \ft^+=\{H\in \ft:\text{$\gamma_j(H)=0$ for all $j=1,\dots,r$}\}\,.$$
So $\ft=\ft^-\oplus \ft^+$.

Let $\overline{X}$ denote the conjugate of $X\in\frakg_\C$ with respect to $\frakg$.  For every $\gamma \in \Delta_n$ choose $X_\gamma\in \frakg\setminus \{0\}$ so that $\overline{X_\gamma}
=X_{-\gamma}$ and $\gamma([X_\gamma, X_{-\gamma}])=2$.
Then $\fa=\sum_{j=1}^r  \R (X_{\gamma_j}+X_{-\gamma_j})$ is a maximal abelian subspace of
$\fp$ and $\fj=\ft^+ \oplus \fa$ is a Cartan subspace of $\frakg$.

Let $c=\Ad\Big(\exp\frac{\pi}{4}\sum_{j=1}^r (X_{\gamma_j}-X_{-\gamma_j})\Big)$. Then $c$ is an automorphism of $\frakg_\C$ mapping $i\ft^-$ onto $\fa$ and fixing $\ft^+$.
Let $c_*$ denote the adjoint of $c^{-1}:\fj_\C\to \ft_\C$. Then $c_*(\Delta)$ is the root system
of $(\frakg,\fj)$. The nonzero restrictions of $c_*(\Delta^+)$ to $\fa$ form a positive system $\Sigma^+$ of the (restricted) root system $\Sigma$  of $(\frakg,\fa)$.
The root system $\Sigma$ is the one appearing in the table of Remark \ref{rem:hermitian}.

\subsubsection{Homogeneous line bundles}
\label{subsection:hermitian-analysis}


We keep the notation above.
For $\ell \in\C$ define $\tau_\ell:K\to \C$ by
$\tau_\ell(k)=1$ for $k\in K_s$ and $\tau_\ell(\exp(tZ))=e^{it\ell}$, where $Z$ is a suitably chosen basis of $\fz$ and $t\in \R$. All one-dimensional representations of $K$ are of this form. The representation $\tau_\ell$ is unitary if and only if
$\ell\in \R$. Let $G_\C$ denote the simply connected complexification of $G$ and let
$K_\C$ denote the connected subgroup of $G_\C$ complexifying $K$. If $G_\R$ and $K_\R$ respectively denote the analytic subgroups of $G_\C$ with Lie algebras $\mathfrak g$ and $\mathfrak k$, then $G/K=G_\R/K_\R$, but
$\tau_\ell$ is a representation of $K_\R$ if and only if $\ell \in \Z$.

Let $E_\ell$ denote the homogeneous line bundle on $G/K$ associated with $\tau_\ell$. The space of smooth sections of $E_\ell$ can be identified with the space $C^\infty(G/K;\tau_\ell)$ of
$C^\infty$ functions on $G$ such that $f(gk)=\tau_\ell(k)^{-1} f(g)=\tau_{-\ell}(k) f(g)$ for all $g\in G$ and $k\in K$. We denote by $C^\infty_\ell(G)$ the subspace of $C^\infty(G/K;\tau_\ell)$ of functions $f$ on $G$ satisfying $f(k_1 g k_2)=\tau_{-\ell}(k_1)f(x)\tau_{-\ell}(k_2)=\tau_{-\ell}(k_1k_2)f(x)$ for all $k_1,k_2\in K$.
For instance, if $\ell=0$, then $\tau_0$ is the trivial representation. Hence $C^\infty(G/K;\tau_0)$ and $C^\infty_0(G)$ can be respectively identified with the space $C^\infty(G/K)$ of smooth functions on $G/K$ and the space $C^\infty(G/K)^K$ of $K$-left-invariant functions on $G/K$.

Let $A$ be the connected subgroup of $G$ of Lie algebra $\fa$. Then $\fa$ is abelian and $\exp:\fa \to A$ is a diffeomorphism. Let $\fa^+$ be the positive Weyl chamber in $\fa$ associated with $\Sigma^+$, and set $A^+=\exp(\fa^+)$.  Let $M$ and $M^*$ be the centralizer and the normalizer of $A$ in $K$, respectively. Recall that they have the same Lie algebra, indicated by $\fm$. The Weyl group $W$ of $\Sigma$ agrees with $M^*/M$ acting by conjugation on $A$.
By the decomposition $G=KAK$,  every element $g\in G$ can be written as $g=k_1ak_2$ with $A$-component $a$ unique up to Weyl group conjugates.  Recall that $k_1ak_2\in A$ for all $a\in A$ if and only if $k_1=k_2^{-1} \in M^*$. Then, on $A$, the condition that $f \in C^\infty_\ell(G)$ becomes that for every $w=yM \in W=M^*/M$ and $a \in A$ we have $f(wa)=f(y a y^{-1})=\tau_\ell(e)f(a)=f(a)$. Here $e$ denotes the identity element of $G$. Thus, the restriction $f|_A$ of $f \in C^\infty_\ell(G)$ to $A$ belongs to $C^\infty(A)^W$, the space of $W$-invariant smooth functions on $A$. By continuity and by the decomposition
$G=K\overline{A^+}K$, an element $f\in C^\infty_\ell(G)$ can be also identified by its restriction $f|_{A^+}$ to $A^+$.

\section{Invariant differential operators on homogeneous line bundles}
\label{section:diff-l-on-G/K}

Let $G/K$ be an Hermitian Riemannian symmetric space of the noncompact type and let
$\ell\in \Z$ be fixed. We keep the notation of the subsections \ref{subsection:hermitian} and
\ref{subsection:hermitian-analysis}.

Let $\D_\ell(G/K)$ denote the algebra of all left-invariant differential operators on $G$ mapping $C^\infty(G/K;\tau_\ell)$ into itself. For $\ell=0$, $\D_0(G/K)$ concide with the commutative algebra $\D(G/K)$ of left-invariant differential operators on $G/K$.

Let $\fUg$ denote the universal enveloping algebra of $\frakg_\C$.
If $\fs$ is a subspace of $\frakg$, we denote by $S(\fs_\C)$ the symmetric algebra over $\fs_\C$ and
set $\fUs=\lambda(S(\fs_\C))$, where $\lambda$ is the symmetrization map. Then $\fUs$ is
the subalgebra of $\fUg$ generated by $1$ and $\fs$ if $\fs$ is a subalgebra of $\frakg$.
The elements of $\fUg$ can be naturally identified with left-invariant differential operators on $G$ by
$$(X_1\cdots X_kf)(g)=\Big.\frac{\partial^k}{\partial t_1 \cdots \partial t_k}
f\big(g \exp(t_1 X_1) \cdots \exp(t_k X_k)\big)\Big|_{t_1=\dots=t_k=0}\,.$$

Let $\fUg^K$ be the space of $\Ad(K)$-invariant elements in $\fUg$ and set $\fk_\ell=\{X+\tau_\ell(X):X\in \fk\}$.  Under the the above identification, $\fUg^K$ is mapped homomorphically onto
$\D_\ell(G/K)$, with kernel $\fUg^K \cap \fUg\fk_\ell$. See e.g. \cite[Theorem 2.1]{ShimenoEigenfunctions}. The isomorphism in the following lemma goes under the name of Harish-Chandra isomorphism.

\begin{Lemma} \label{lemma:Dell-poly}
$\D_\ell(G/K)\cong \fUg^K/\fUg^K \cap
\fUg\fk_\ell$ is isomorphic to the algebra $S(\fa_\C)^W$.  In particular, $\D_\ell(G/K)$ is a commutative algebra of rank $r$.
\end{Lemma}
\begin{proof}
This is \cite[Theorem 2.4]{ShimenoEigenfunctions}.
\end{proof}

Let $\Areg=A \cap \Hreg$ denote the open submanifold of regular elements in $A$.
The enveloping algebra $\fUa$ can be identified with the space
$S(\fa_\C)$ of polynomial differential operators in $A$ (or $\fa$). Their elements can be thought as belonging to $\D(\Areg)$, the algebra of differential operators on $\Areg$.
The \textit{$\tau_{-\ell}-$radial component} of $D\in \D_\ell(G/K)$ is the differential operator $\delta_\ell(D)$ on $\Areg$ which satisfies $(Df)|_{\Areg}=\delta_\ell(D)(f|_{\Areg})$ for all
$f \in C^\infty_\ell(G)$.
The Laplacian of the line bundle $E_\ell$ is the element of $\D_\ell(G/K)$ corresponding under
the Harish-Chandra isomorphism to the (equivalence class of) the Casimir operator $\Omega\in\fUg^K$.
Its $\tau_{-\ell}$-radial component $\delta_\ell(\Omega)$ has been computed (with different methods) in \cite[Lemma 2.4]{ShimenoPlancherel} and in \cite[Theorem 5.1.7]{HS}. It is given by
\begin{equation}
\label{eq:radcompCasimir}
\delta_\ell(\Omega)=L(m)+\frac{p^2\ell^2}{4} \sum_{j=1}^r
 \frac{1}{\cosh^2\big(\frac{\beta_j}{2}\big)}+\tau_{-\ell}(\Omega_\fm)\,.
\end{equation}
In \eqref{eq:radcompCasimir},
$m=(m_\rml,m_\rmm,m_\rms)$ denotes the multiplicity function of $G/K$,
$L(m)$ is as in \eqref{eq:Laplm}, $p$ is the norm of the roots $\beta_j$, and $\Omega_\fm$ is the Casimir operator of $\fm_\C$.
In Case I, $\fm\subset \fk_s$ (the Lie algebra of $K_s$). Hence $\tau_{-\ell}(\Omega_\fm)=0$ in this case.

If $\ell=0$, then the image of $\Omega$ in $\D(G/K)$ is the Laplace-Beltrami operator $L_{G/K}$
of the Riemannian space $G/K$. Then  $\delta_0(\Omega)=L(m)$ is the radial part of $L_{G/K}$
for the $K$-action on $G/K$ with $A^+\cdot o$ as a transversal manifold. Here $o=eK$, with $e$ the unit element of $G$, is the base point of $G/K$.

We define $\mathcal{R}_0$ as the subalgebra of $\C[\Hreg]$ generated by $1$,  the functions
\begin{equation}
\label{eq:R0}
f_\a=e^{-\a}(1-e^{-2\a})^{-1}  \qquad (\a\in \Sigma^+)
\end{equation}
and the functions $g_\a$, as in \eqref{eq:galpha}, with $\a\in \Sigma^+$.
Then $\mathcal{R}_0$ contains the algebra $\mathcal{R}$. Furthermore,
$g_\a\pm f_\a=(1\mp e^{-\alpha})^{-1}$ and $(1+e^{-\alpha})^{-1}(1- e^{-\alpha})^{-1}=g_\a$.
Hence $\mathcal{R}_0$ coincides with the algebra considered by Harish-Chandra in \cite[p. 63-64]{HC60}.
The $f_\a$ and $g_\a$ can be considered as functions on $\Areg$ by setting for $h\in \Areg$:
$$f_\a(h)=f_\a(\log h) \qquad \text{and} \qquad g_\a(h)=g_\a(\log h)\,.$$
Here $\log:A\to \fa$ is the inverse of $\exp$.

For every integer $d\geq 1$, let $\mathcal{R}_{0,d}$ be the linear span of the monomials in $f_\a$ and $g_\a$ of degree $d$. Set
$$
\mathcal{R}_0^{(d)}=\sum_{h=1}^d \mathcal{R}_{0,h}
\qquad \text{and} \qquad \mathcal{R}_0^{+}=\sum_{h=1}^\infty
\mathcal{R}_{0,h}\,.
$$
Recall that an element $b\in \fUg$
is said to be homogeneous of degree $d$, written $\deg(b)=d$, if
$b\in \l(S_d(\frakg_\C))$, where $S_d(\frakg_\C)$ denotes the space of
homogenous elements in $S(\frakg_\C)$ of degree $d$ and $\lambda$ is
the symmetrization map.

Set $\fn=\sum_{\alpha\in \Sigma^+} \frakg^\alpha$, where $\frakg^\alpha$ denotes the root space of $\a\in \Sigma$.
The direct sum $\frakg=\fk\oplus \fa\oplus \theta(\fn)$, where $\theta$ is the Cartan involution, yields the
direct sum decomposition
\begin{equation}
\label{eq:decomposition-Ug-1}
\fUg=\fUa\fUk\oplus \theta(\fn_\C)\fUg\,.
\end{equation}
Let $\kappa:\fUg\to \fUa\fUk$ denote the
corresponding projection.
Let $\mathfrak{q}$ be the orthogonal complement in $\fk$ (with respect to the Killing form)
of the centralizer of $\fa$.
For $h \in \Areg$ and an element $b \in \fUg,$
$b^{h}$ will denote the operator $\Ad(h)b.$

Harish-Chandra proved the following proposition. See \cite[Lemmas
7 and 10]{HC60}; see also \cite[Proposition 4.1.4]{GV}. For any
unexplained notation we refer to \cite{HC60}.

\begin{Prop}
\label{prop:radial1} Let $b\in \fUg$ be homogeneous of
degree $d$ and write $\kappa(b)=\sum_{i=1}^m v_i \xi''_i$ with
$v_i \in \fUa$, $\xi''_i\in \fUk$. Then there exists a
finite number of elements $\xi_j\in \fUq$, $\xi'_j\in
\fUk$,  $u_j\in \fUa$ and $f_j\in \mathcal{R}_0^+$, with
$1\leq j\leq n$, such that
\begin{enumerate}
\thmlist
\item $\deg(u_j)<d$ and $f_j \in \mathcal{R}_0^{(m-\deg(u_j))}$ for all $1\leq j\leq n$,
\item $\deg(\xi_j)+\deg(u_j)+\deg(\xi_j')\leq d$ for all $1\leq j\leq n$,
\item $b=\sum_{i=1}^m v_i \xi''_i+\sum_{j=1}^n f_j(h)\xi_j^{h^{-1}} u_j \xi'_j$ for all $h\in \Areg$.
\end{enumerate}
\end{Prop}

If $f \in C_\ell^\infty(G)$, $Z \in \fk_\C$ and $h\in \Areg$ is fixed, then
\begin{equation}
\label{eq:fk-on-Cell}
(Zf)(h)=\tau_{-\ell}(Z)f(h)\quad \text{and} \quad (Z^{h^{-1}}f)(h)=\tau_{-\ell}(Z)f(h)\,.
\end{equation}
Hence
\begin{eqnarray*}
(b f)(h)&=&\Big[\Big(\sum_{i=1}^m v_i \xi''_i+\sum_{j=1}^n f_j(h)\xi_j^{h^{-1}} u_j \xi'_j\Big)f\Big](h)\\
&=&\sum_{i=1}^m \tau_{-\ell}(\xi''_i)(v_i f)(h)+\sum_{j=1}^n f_j(h)\tau_{-\ell}(\xi_j \xi'_j)
(u_jf)(h)\,.
\end{eqnarray*}
The $\tau_{-\ell}$-radial component of $b$ is therefore
$$
\delta_\ell(b)=\sum_{i=1}^m \tau_{-\ell}(\xi''_i)v_i+\sum_{j=1}^n f_j\,\tau_{-\ell}(\xi_j \xi'_j)
u_j\,.
$$
This shows that the coefficients of $\delta_{\ell}(b)$ are in
$\mathcal{R}_0$. In fact, it turns out that they are in a proper
subalgebra of $\mathcal{R}_0$. For the case $\ell=0$, it is
classical (and proven in \cite[Lemma 24]{HC58} for arbitrary
semisimple $G$) that they are in $\mathcal{R}$. In fact,
$\mathcal{R}$ suffices for the algebras
$C_\ell^\infty(G)$ for all $\ell\in \Z$, under the condition,
given in section \ref{section:notation}, that a root system of
type $C_r$ is considered as a root system of type $BC_r$ with
$m_\rms=0$. This result was stated without proof
in \cite[Theorem 5.1.4]{HS}.
The explicit expression given in \eqref{eq:radcompCasimir} proves
it for the Laplace operator. Indeed, in the $BC_r$ case, the
algebra $\mathcal{R}$ is generated by $1$, the functions
$g_{\frac{\beta_j}{2}}=(1-e^{-\beta_j})^{-1}$ and
$g_{\beta_j}=(1-e^{-2\beta_j})^{-1}$ with $1\leq j \leq r$, and
the functions $g_{\frac{\beta_i \pm \beta_j}{2}}=(1-e^{-(\beta_i
\pm\beta_j)})^{-1}$ with $1\leq j < i \leq  r$. By
\eqref{eq:coth-alpha}, $\mathcal{R}$ is also generated by $1$, the
functions  $\coth(\frac{\beta_j}{2})$ and $\coth(\beta_j)$ with
$1\leq j \leq r$, and the functions $\coth(\frac{\beta_i
\pm\beta_j}{2})$ with $1\leq j < i \leq  r$. Because of the
identity
$$
\big[\cosh(\tfrac{\beta}{2})\big]^{-2}=4\coth^2\beta- \coth^2(\tfrac{\beta}{2})+1\,,
$$
the operator $\delta_\ell(\Omega)$ has coefficients in $\mathcal{R}$.

For general differential operators in $\D_\ell(G/K)$, proving that
their $\tau_{-\ell}$-radial component has coefficients in
$\mathcal{R}$ requires more work. We prove it below, using the
specific properties of the representation $\tau_\ell,$ and
Harish-Chandra's inductive procedure used to prove Proposition
\ref{prop:radial1}.

For $\alpha\in \Sigma$  choose a basis $\{X_{\alpha,1},\dots, X_{\alpha,m_\alpha}\}$
of the root space $\frakg^\alpha$ so that $X_{\alpha,k}=-\theta X_{\alpha,k}$ and
$B(X_{\alpha,k}, X_{\alpha,-h})=\delta_{kh}$ for all $k,h=1,\dots, m_\alpha$.
Write
$$X_{\alpha,k}=Z_{\alpha,k}+Y_{\alpha,k}  \qquad (k=1,\dots,m_\alpha)$$
with $Z_{\alpha,k}\in \fk$ and $Y_{\alpha,k}\in \fp$.

\begin{Lemma}
\label{lemma:vanishing-tauell}
Keep the above notation.
Then $\tau_{-\ell}(Z_{\alpha,k})=0$ if $\alpha\in \Sigma_\rms\cup \Sigma_\rmm$.
\end{Lemma}
\begin{proof}
See the proof of \cite[Lemma 2.4]{ShimenoPlancherel} and \cite[\S
5]{Schl84}. For $\alpha\in \Sigma_\rmm$ see also \cite[Lemma
7.1]{ShimenoEigenfunctions} and \cite[Remark
7.2]{ShimenoEigenfunctions}.
\end{proof}

\begin{Thm} \label{thm:coefficients-taul-radial}
For each $D\in \D_\ell(G/K)$ there exists a unique operator $\delta_\ell(D) \in \D_\mathcal{R}^W=(\mathcal{R}\otimes S(\fa_\C))^W$ such that
$$Df|_A=\delta_\ell(D) f|_A$$
for all $f\in C^\infty_\ell(G)$. The map $\delta_\ell: \D_\ell(G/K)\to \D_\mathcal{R}^W$ is an injective algebra homomorphism.
\end{Thm}
\begin{proof}
(cf. \cite[Theorem 5.1.4]{HS} and \cite[\S 4]{HC60}). \;
The only point that one cannot find in the references
is that $\delta_\ell$ maps $\D_\ell(G/K)$ into
$\D_\mathcal{R}$, i.e. that the $\tau_{\ell}$-radial components have coefficients in $\mathcal{R}$ rather than in $\mathcal{R}_0$.

We follow the inductive procedure used to prove Proposition
\ref{prop:radial1}, written in  terms of restricted roots spaces.
The functions $f_j$ of Proposition \ref{prop:radial1} are
constructed inductively from \eqref{eq:decomposition-Ug-1}: since
$b-\kappa(b)\in \theta(\fn_\C)\fUg$, there are elements
$b_{\alpha,k}\in \fUg$ with $\deg(b_{\alpha,k})<d$ so that
$$b-\kappa(b)=\sum_{\alpha\in \Sigma^+}\sum_{k=1}^{m_\alpha}
\theta(X_{\alpha,k})b_{\alpha,k}\,.$$ As before, write
$X_{\alpha,k}=Z_{\alpha,k}+Y_{\alpha,k}$ with $Z_{\alpha,k}\in
\fk$ and $Y_{\alpha,k}\in \fp$.
Then, according to \cite[Lemma
6]{HC60}, for all $h=\exp(H)\in \Areg$ we have
\begin{equation*}
\theta(X_{\alpha,k})=2\frac{e^{\beta(H)}}{e^{2\beta(H)}-1}Z_{\alpha,k}^{h^{-1}}-2\frac{1}{e^{2\beta(H)}-1}Z_{\alpha,k}
=2f_\alpha(h)Z_{\alpha,k}^{h^{-1}}+2(1-g_\alpha)(h)Z_{\alpha,k}\,.
\end{equation*}
Hence
\begin{equation}
\label{eq:thetaXbetabbeta}
\theta(X_{\alpha,k})b_{\alpha,k}
=2f_\alpha(h)Z_{\alpha,k}^{h^{-1}}b_{\alpha,k}+2(1-g_\alpha)(h)b_{\alpha,k} Z_{\alpha,k}+2(1-g_\alpha)(h)[Z_{\alpha,k}, b_{\alpha,k}]\,,
\end{equation}
where $\deg(b_{\alpha,k})<d$ and $\deg([Z_{\alpha,k}, b_{\alpha,k}])<d$.
If $\tau_{\ell}(Z_{\alpha,k})=0$, then by \eqref{eq:fk-on-Cell}  the first two operators on the left-hand side of \eqref{eq:thetaXbetabbeta} do not contribute to the $\tau_{\ell}$-radial component of $b$.
By this and Lemma \ref{lemma:vanishing-tauell}, we deduce that the coefficients of the
$\tau_{\ell}$-radial component of $b$ are in the algebra generated by $1$, $g_\alpha$ (with $\alpha\in \Sigma^+$) and $f_{\beta_j}$ (with $j=1,\dots,r$). Since $f_{\beta_j}=g_{\beta_j/2}-g_{\beta_j}$, this algebra is $\mathcal{R}$.
\end{proof}

Given the multiplicity function $m=(m_\rms,m_\rmm,m_\rml=1)$ and $\ell\in \Z$, define the deformed multiplicity function $m(\ell)$ on $\Sigma$ by
\begin{equation}
\label{eq:mult-l}
m_\a(\ell)=\begin{cases}
m_\rms+2\ell  &\text{if $\a\in \Sigma_\rms$}\\
m_\rmm &\text{if $\a\in \Sigma_\rmm$}\\
1-2\ell  &\text{if $\a\in \Sigma_\rml$}\,.
\end{cases}
\end{equation}
Then
\begin{equation}
\label{eq:rhol}
 \rho(m(\ell))=\frac{1}{2} \sum_{\a \in \Sigma^+} m_\a(\ell) \a=
 \rho(m)-\frac{\ell}{2}\, \sum_{j=1}^r \beta_j\,.
\end{equation}
For $a=\exp(x)\in A$ define:
\begin{equation}
\label{eq:u}
u(a)=u(x)=\prod_{j=1}^r  \cosh\big(\frac{\beta_j(x)}{2}\big)\,.
\end{equation}
In \cite[Proposition 2.6]{ShimenoPlancherel}, Shimeno proved that
\begin{equation}
\label{eq:radial-cong-h1}
u^\ell\circ \Big(\delta_\ell(\Omega)+\inner{\rho(m)}{\rho(m)}-\tau_{-\ell}(\Omega_\fm)\Big)
\circ u^{-\ell}=L(m(\ell))+\inner{\rho(m(\ell))}{\rho(m(\ell))}\,.
\end{equation}

\begin{Lemma}
\label{lemma:conj-by-ul}
Conjugation by $u^{-\ell}$ is an algebra isomorphism of $\D_\mathcal{R}$ mapping  $\D_\mathcal{R}^W$ onto itself. In particular, for every multiplicity function $m$,
$u^{-\ell}\circ \D(m)\circ u^{\ell}$ is a commutative subalgebra of $\D_\mathcal{R}^W$
of rank $r$. It is the centralizer of $u^{-\ell} \circ\delta_\ell(L(m))\circ u^{\ell}$ in $\D_\mathcal{R}^W$.
\end{Lemma}
\begin{proof}
Proving that the map  $D\to u^{-\ell}\circ D\circ u^{\ell}$ defines an algebra automorphism
of $\D_\mathcal{R}$ boils down to showing that if $\xi\in \fa$ then $u^{-\ell}\circ \partial_\xi\circ u^{\ell}\in \D_\mathcal{R}$. For $f\in C^\infty(A)$ (or $C^\infty(\fa)$), by Leibniz' rule, we have
\begin{equation}
\label{eq:conj by uell-x}
(u^{-\ell}\circ \partial_\xi\circ u^{\ell})(f)=u^{-\ell} \partial_\xi(u^{\ell}f)=u^{-\ell} \partial_\xi(u^{\ell})f+ \partial_\xi(f)
          = \big(\ell\, u^{-1} \partial_\xi(u)+ \partial_\xi\big)(f)\,,
\end{equation}
where
\begin{equation}
\label{eq:x-log-ul}
u^{-1} \partial_\xi(u)=\, \sum_{j=1}^r \frac{\b_j(\xi)}{2} \tanh\big(\frac{\b_j}{2}\big) \in \mathcal{R}
\end{equation}
because
$$\tanh\big(\frac{\b_j}{2}\big)=\coth(\b_j)-\frac{1}{\sinh(\b_j)}=4g_{\b_j}-2g_{\b_j/2}-1\,.$$

 To show that the conjugation by $u^{-\ell}$ maps
$\D_\mathcal{R}^W$ into itself, recall that $W$ acts on the roots
by permutations and sign changes, so $\Sigma_\rms$ is a single
$W$-orbit. Since $\cosh x$ is even, the function $u$ (and so
$u^{\pm \ell}$) is $W$-invariant. Let $D\in \D_\mathcal{R}^W$. For
all $w\in W$ and $f\in C^\infty(\fa)$ (or $C^\infty(A)$), we have
therefore 
\begin{eqnarray*}
w(u^{-\ell} \circ D \circ u^{\ell})(wf)&=& w\big((u^{-\ell} \circ D \circ u^{\ell})f\big)\\
&=& (wu^{-\ell}) w\big((D \circ u^{\ell})f\big)\\
&=& u^{-\ell}(wD)\big(w(u^{\ell}f)\big)\\
&=& u^{-\ell} D\big((u^{\ell}(wf)\big)\\
&=& (u^{-\ell} \circ D \circ u^{\ell})(wf)\,.
\end{eqnarray*}
Thus $w(u^{\ell} \circ D \circ u^{-\ell})=u^{\ell} \circ D \circ u^{-\ell}$.
The last part of the lemma follows from the first part together with Harish-Chandra's isomorphism and Lemma \ref{lemma:commutatorL}.
\end{proof}

Theorem \ref{thm:coefficients-taul-radial} and the previous lemma allow us to prove the following result from \cite{HS}.

\begin{Cor}
\label{cor:conj-by-ul}
$\delta_\ell$ is an algebra isomorphism of $\D_\ell(G/K)$ onto the subalgebra
$u^{-\ell} \circ \D(m(\ell))) \circ u^{\ell}$ of $\D_\mathcal{R}^W$. Moreover,
$\delta_\ell\big(\D_\ell(G/K) \big)=u^{-\ell} \circ \D(m(\ell))) \circ u^{\ell}$ is the centralizer in $\D_\mathcal{R}^W$ of $\delta_\ell(\Omega)$.
\end{Cor}
\begin{proof}
By Lemmas \ref{lemma:commutatorL} and \ref{lemma:conj-by-ul},
$u^{-\ell} \circ \D(m(\ell))) \circ u^{\ell}$ is a commutative
subalgebra of $\D_\mathcal{R}^W$ which is equal to the centralizer
of $u^{-\ell} \circ  L(m(\ell)) \circ u^{\ell}$. Hence, by
\eqref{eq:radcompCasimir}, it is the centralizer in
$\D_\mathcal{R}^W$ of $\delta_\ell(\Omega)$. Because of Theorem
\ref{thm:coefficients-taul-radial}, $\delta_\ell\big(\D_\ell(G/K)
\big)$ is a commutative subalgebra of $\D_\mathcal{R}^W$
containing $\delta_\ell(\Omega)$. Hence
$\delta_\ell\big(\D_\ell(G/K) \big)$ is contained in the
centralizer of $u^{-\ell} \circ L(m(\ell)) \circ u^{\ell}$.
Furthermore, by Lemma \ref{lemma:Dell-poly}, by the Harish-Chandra
isomorphism and the fact that $\delta_\ell$ and the conjugation by
$u^{-\ell}$ are algebra isomorphisms, these two algebras have the
same rank. Thus they coincide.
\end{proof}

\section{$\tau_{-\ell}$-Hypergeometric functions}
\label{section:new-cherednik}

Let $(\fa,\Sigma,m,\ell)$ be a quadruple consisting of a
$r$-dimensional Euclidean real vector space $\fa$, a root system
$\Sigma$ of type $BC_r$ in $\fa^*$ and a complex multiplicity
function $m=(m_\rms,m_\rmm,m_\rml=1)$ on $\Sigma$, as in section
\ref{subsection:roots}, together with a complex number $\ell \in
\C$. Unless stated otherwise, we shall keep the notation
introduced in the previous sections. In particular, we denote by
$m(\ell)$ the $1$-parameter deformation of $m$ defined in
\eqref{eq:mult-l}. Notice that many of the constructions below just use that the commutative algebra we are using is of the form $u^{-\ell}\circ \D(m(\ell))\circ u^\ell$. So they could be done without the assumption $m_\rml=1$. In this
case, of course, $m(\ell)$ should be given by \eqref{eq:mult-l-gen}.

For $\xi\in \fa$ define the Cherednik operators $T_{\ell,\xi}\in \D_{\mathcal{R}} \otimes \C[W]$ by
\begin{equation}
\label{eq:Tellx} T_{\ell,\xi}(m)= \frac{\ell}{2}  \sum_{j =1 }^r
\beta_j(\xi) \tanh \Big( \frac{\beta_j}{2}\Big)  + \partial_\xi -
\rho(m(\ell);\xi) + \sum_{\alpha \in \Sigma^+} m_\alpha(\ell)
\alpha(\xi) (1-e^{-2\alpha})^{-1} \otimes (1-r_\alpha).
\end{equation}
Let $u$ be the function defined by \eqref{eq:u}. A simple computation (using that $u$ commutes with $1-r_\alpha$ for $\a\in \Sigma$) shows that
$$
T_{\ell,\xi}(m)= T_\xi(m(\ell)) + \ell u^{-1} \partial_\xi(u)
=u^{-\ell} \circ T_\xi(m(\ell)) \circ u^{\ell},
$$
where $T_\xi(m(\ell))$ is the Cherednik operator, defined in \eqref{eq:T}, corresponding to $m(\ell)$ of the multiplicity function $m$.
It follows that
$\{T_{\ell,\xi}(m) : \xi \in \fa\}$ is a commutative family of differential-reflection operators.
So, the map $\xi \to T_{\ell,\xi}(m)$ extends uniquely to an algebra homomorphism $p \to
T_{\ell,p}(m)$ from $\polya$ to $\D_{\mathcal{R}} \otimes \C[W]$ such that
$T_{\ell,p}(m)=u^{-\ell} \circ T_p(m(\ell)) \circ u^{\ell}$.
In particular, one can define the $\tau_{-\ell}-$Heckman-Opdam Laplacian
\begin{equation}
\label{eq:tellHOLaplacian}
T_{\ell, p_L}(m)=\sum_{j=1}^r T_{\ell,\xi_j}(m)^2= u^{-\ell}\circ T_{p_L}(m(\ell)) \circ u^{\ell}\,,
\end{equation}
where $\{\xi_j\}_{j=1}^r$ is any orthonormal basis of $\fa$ and $p_L$ is defined by $p_L(\l):=\inner{\l}{\l}$ for $\l \in \frakacs$.

Let
\begin{equation}
\label{eq:beta-taul-HOLaplacian}
L_\ell(m)
=L(m)+\frac{\ell^2 p^2}{4} \sum_{j=1}^r \frac{1}{\cosh^2(\frac{\b_j}{2})}\,.
\end{equation}
In particular, if $m=(m_\rms,m_\rmm,m_\rml=1)$ is geometric, then
\begin{equation}
\label{eq:Lell-deltaell}
L_\ell(m)=\delta_\ell(\Omega)-\tau_{-\ell}(\Omega_\fm)\,.
\end{equation}
See \eqref{eq:radcompCasimir}.
Let $m=(m_\rms,m_\rmm,m_\rml=1)$ be arbitrary.  Apply \cite[Corollary 2.1.2]{HS} (using our symmetric space notation) with $m$ and $k$ respectively replaced by our $m(\ell)$ and $m$ and with $l=(l_\rms,l_\rmm,l_\rml)$ so that $l_\rmm=0$ and $l^2_\rms=-l^2_\rml=-4\ell^2$.
We obtain:
\begin{equation}
\label{eq:L-Lell}
L_\ell(m)+\inner{\rho(m)}{\rho(m)}= u^{-\ell} \circ \Big( L(m(\ell))+\inner{\rho(m(\ell))}{\rho(m(\ell))} \Big)\circ u^{\ell}\,.
\end{equation}
This together with \eqref{eq:HOLaplacian} yields the explicit expression
\begin{eqnarray}
\label{eq:taul-HOLaplacian}
T_{\ell, p_L}(m)
&=&L_\ell(m)+\inner{\rho(m)}{\rho(m)}
-\sum_{\a\in\Sigma^+} m_\a(\ell) \frac{\inner{\a}{\a}}{\sinh^2 \a} \otimes (1-r_\a) \\
\label{eq:taul-tau0-HOLaplacian}
&=&T_{p_L}(m)+\frac{\ell p^2}{4} \sum_{j=1}^r \frac{1}{\cosh^2(\frac{\b_j}{2})} \otimes
\big( \ell -2(1-r_{\beta_j})\big)\,.
\end{eqnarray}
Recall the operator $\beta$ from \eqref{eq:beta}.
If $p\in \polya$, then $\beta\big(T_{\ell,p}(m)\big)=
u^{-\ell}\circ \beta\big(T_{p}(m(\ell))\big) \circ u^{\ell}\in \D_\mathcal{R}$.
Set $D_{\ell,p}(m)=\beta\big(T_{\ell,p}(m)\big)$ for $p\in \polya^W$. For instance,
$D_{\ell,p_L}(m)=L_\ell(m)+\inner{\rho(m)}{\rho(m)}$. Furthermore, set
\begin{equation}
\label{eq:Dlm}
\D_\ell(m)=\{D_{\ell,p}(m): p\in \polya^W\}\,.
\end{equation}

\begin{Prop}
In the above notation, $\D_\ell(m)=u^{-\ell}\circ \D(m(\ell)) \circ u^{\ell}$ is a commutative subalgebra of $\D_\mathcal{R}^W$ of rank $r$. It is the centralizer of $L_\ell(m)$
in $\D_\mathcal{R}^W$. As a consequence, $\D_\ell(m)=\D_{-\ell}(m)$.
Furthermore, if $m=(m_\rms,m_\rmm,m_\rml=1)$ is geometric and $\ell\in \Z$, then $\D_\ell(m)=\delta_{\ell}(\D_\ell(G/K))$.
\end{Prop}
\begin{proof}
The first part is an immediate consequence of Lemma \ref{lemma:conj-by-ul} and Corollary  \ref{cor:conj-by-ul}. The equality $\D_\ell(m)=\D_{-\ell}(m)$ is a consequence of the
fact that $L_\ell(m)=L_{-\ell}(m)$ by \eqref{eq:beta-taul-HOLaplacian}.
\end{proof}

Let $\l\in \fa^*_\C$ be fixed. The \textit{$\tau_{-\ell}-$Heckman-Opdam
hypergeometric function with spectral parameter $\l$} is the unique
$W$-invariant analytic solution $F_{\ell,\l}(m)$ of the system of
differential equations 
\begin{equation}
  \label{eq:hypereq}
D_{\ell,p}(m) f=p(\l)f \qquad (p \in \polya^W),
\end{equation}
which satisfies $f(0)=1$.

The \textit{non-symmetric $\tau_{-\ell}-$hypergeometric function with
spectral parameter $\l$} is the unique analytic solution
$G_{\ell,\l}(m)$ of the system of differential equations 
\begin{equation}
  \label{eq:hypereq}
T_{\ell,\xi}(m) g=\l(\xi) g \qquad (\xi \in \fa),
\end{equation}
which satisfies $g(0)=1$.

The symmetric and non-symmetric $\tau_{-\ell}-$Heckman-Opdam
hypergeometric functions are therefore (suitably normalized) joint
eigenfunctions of the commutative algebras $\D_\ell(m)$ and
$\{T_{\ell,p}:p\in S(\fa_\C)\}$, respectively. Notice that the
equality $\D_\ell(m)=\D_{-\ell}(m)$ yields
\begin{equation}
\label{eq:Fellm-even-ell}
F_{\ell,\l}(m)=F_{-\ell,\l}(m)\qquad (\l\in\fa_\C^*)\,.
\end{equation}
On the other hand, $G_{\ell,\l}(m)\neq G_{-\ell,\l}(m)$ in general (as one can see
from the rank-one example at the end of this section).
By definition, $F_{\ell,\l}(m;x)$ is $W$-invariant in $x\in \fa$ and in $\l\in \fa^*_\C$.
Furthermore, since $u^\ell\circ D_{\ell,p}(m)\circ u^{-\ell}=D_{p}(m(\ell))$ and
$u^\ell\circ T_{\ell,\xi}(m)\circ u^{-\ell}=T_{\xi}(m(\ell))$, one obtains for all $\l\in \fa_\C^*$:
\begin{eqnarray}
\label{eq:Fell-F}
F_{\ell,\l}(m)&=&u^{-\ell} F_{\l}(m(\ell))\,,\\
G_{\ell,\l}(m)&=&u^{-\ell} G_{\l}(m(\ell))\,.
\label{eq:Gell-G}
\end{eqnarray}
As in the case $\ell=0$,
\begin{equation}
\label{eq:Fell-Gell}
F_{\ell,\l}(m;x)=\frac{1}{|W|}\, \sum_{w\in W} G_{\ell,\l}(m;w^{-1}x) \qquad (x \in \fa)\,.
\end{equation}
As a consequnce of \eqref{eq:Fell-F} and of the corresponding properties of the Heckman-Opdam hypergeometric functions (see e.g. \cite[Theorem 4.4.2]{HS}, \cite[Theorem 3.15]{OpdamActa}), the
$F_{\ell,\l}(m;x)$ and $G_{\ell,\l}(m;x)$ are entire functions of $\lambda \in \fa_\C^*$, analytic functions in $x\in \fa$
and meromorphic functions of $m=(m_\rms,m_\rmm,m_\rml=1) \in \C^2$.  Let $\mathcal M_{\C,0}=\{m=(m_\rms,m_\rmm,m_\rml)\in \C^3:\Re(m_\rms+m_\rml)\geq 0 \}$.
Observe that $m\in \mathcal M_{\C,0}$ if and only if $m(\ell)\in  \mathcal M_{\C,0}$.  One can show (see Appendix A below) that for fixed $(\l,x)\in \fa^*_\C \times \fa$, the functions $F_\l(m;x)$ and $G_\l(m;x)$ are holomorphic in a neighborhood of $\mathcal M_{\C,0}$. It follows that $F_{\ell,\l}(m;x)$ and $G_{\ell,\l}(m;x)$ are holomorphic near each $m=(m_\rms,m_\rmm,m_\rml=1)\in \mathcal M_0$.

\subsection{The geometric case}
We keep the notation of the subsection \ref{subsection:hermitian}
and section \ref{section:diff-l-on-G/K}.
Let $m=(m_\rms,m_\rmm,m_\rml=1)$ be a geometric multiplicity function, as in Remark \ref{rem:hermitian}. For $\ell\in \Z$ (or $\ell \in \R$ by passing to universal covering spaces) we have $\D_\ell(m)=\delta_\ell(\D_\ell(G/K))$. So $F_{\ell,\l}(m)$
coincides with the restriction to $A\equiv \fa$ of the
$\tau_{-\ell}-$spherical function $\varphi_{\ell,\l}$ with spectral parameter $\l$ on the line bundle $E_\ell$ over $G/K$.  Many authors have studied the $\tau_{-\ell}-$spherical functions in this context. See e.g. \cite[\S 5.2--5.5]{HS}, \cite[\S 6]{ShimenoEigenfunctions}, \cite[\S 3 and
5]{ShimenoPlancherel}, \cite{HoOl14}.
The following proposition summarizes the basic properties of the
the $\tau_{-\ell}-$spherical functions. As in the $K$-biinvariant case, they can be explicitly given by an integral formula, which extends to arbitrary real $\ell$'s the classical integral formula by Harish-Chandra's for the spherical function $\varphi_\l=\varphi_{0,\lambda}$ on $G/K$.

\begin{Prop}
For $\ell \in \mathbb R$ and $\lambda \in \frakacs$
define
\begin{equation}
\label{eq:integral-formula}
 \varphi_{\ell,\lambda}(g) = \int_{K/Z(G)} e^{(\lambda -
\rho)(H(gk))} \tau_{\ell}(\kappa(gk)^{-1}k)~dk, \qquad g \in G.
\end{equation}
The set of functions $\varphi_{\ell, \lambda}$, $\lambda \in
\frakacs$ exhausts the class of (elementary) $\tau_{-\ell}-$spherical functions on $G.$ Two such functions $\varphi_{\ell,\lambda}$
and $\varphi_{\ell,\mu}$ are equal if and only if $\mu =
w\lambda$ for some $w \in W.$ Moreover, for a fixed $g \in G,$
$\varphi_{\ell, \lambda}(g)$ is holomorphic in $(\lambda, \ell)
\in \frakacs \times \mathbb C$. Furthermore, $\varphi_{\ell, \lambda}=\varphi_{-\ell, \lambda}$.
\end{Prop}
\begin{proof}
See \cite[Proposition 6.1]{ShimenoEigenfunctions}.
\end{proof}

As an immediate consequence of the integral formula for $\varphi_{\ell, \l}$ and the properties of Harish-Chandra's spherical functions $\varphi_\l$ ($\ell = 0$ case), we obtain the following corollary.

\begin{Cor}\label{Cor:estimates-spherical}
Let $\ell\in \R,$ then we have:
\begin{enumerate}
\item $\varphi_{\ell, \l}$ is real valued for $\l\in\fa^*$,
\item $|\varphi_{\ell,\l}|\leq \varphi_{\Re\l}$ for $\l\in\fa_\C^*$,
\item $|\varphi_{\ell,\l}(m)|\leq 1$ for $\l \in
C(\rho(m)) + i \fa^\ast$, where $C(\rho(m))$ is the convex hull of
the set $\{w\rho(m): w \in W\}.$
\end{enumerate}
\end{Cor}
\begin{proof}
The first property follows from the equality $\varphi_{\ell,\l}=\varphi_{-\ell,\l}$ because $\tau_\ell+\tau_{-\ell}$ is real valued. Part (2) is a consequence of the integral formula and the
fact that $\tau_\ell$ is unitary. Finally, (3) follows from (2) and the theorem by Helgason and Johnson characterizing the Harish-Chandra's spherical functions which are bounded; see \cite{HeJo}.
\end{proof}

We conclude this subsection by remarking that, because of the factor coming from
$\tau_\ell$, the integral formula \eqref{eq:integral-formula} does not immediately imply that for $\lambda\in \mathfrak{a}^*$ one has $\varphi_{\ell,\l}>0$ unless $\ell=0$. Some sufficient conditions for the positivity will be determined in Lemma \ref{lemma:pos-rank1} for the rank one case and in Corollary \ref{cor:pos-est-ell}, (b), in the general case.

\subsection{The rank-one case}
We speak of rank-one case when $\frak a$ is one dimensional  and
$\Sigma$ is of type $BC_1$. Geometrically this corresponds to
$G=\SU(1,q)$. Then $\Sigma^+=\{\beta/2,\beta\}$.
We identify $\frak a^*$ and $\frak a$ with $\R$ (and their complexifications with $\C$)
by setting $\beta/2\equiv 1$ and $H\equiv 1$, where $H\in \fa$ satisfies $\beta/2(H)=1$.
In particular, this identifies $\rho(m)$ with $\frac{m_\rms}{2}+ m_\rml=\frac{m_\rms}{2}+1$.
The Weyl chamber $\frak a^+$ coincides with the half-line $(0,+\infty)$.
The Weyl group $W$ reduces to $\{-1,1\}$ acting on $\R$ and $\C$ by multiplication.
We normalize the inner product so that $\inner{\beta/2}{\beta/2}=1$.

The functions  $F_\l(m)$ and $G_\l(m)$ can be written in terms of Gauss's hypergeometric functions or, more precisely,
Jacobi's functions; see \cite[p. 90]{OpdamActa}.
The Jacobi function
$$\varphi^{(a,b)}_{i\l}(x)= \hyper{\tfrac{1}{2}(a+b+1-\l)}{\tfrac{1}{2}(a+b+1+\l)}{a+1}{-\sinh^2 x}$$
is defined for all $a,b,\l\in \C$ with $a\neq -1,-2,\dots$ and $x\in \R$, see \cite[(1.1) and (2.4)]{Koorn84}.
Then, with
$$
a=\frac{1}{2}(m_\rms+m_\rml-1)=\frac{1}{2}m_\rms, \qquad b(\ell)=\frac{1}{2}(m_\rml-1)-\ell=-\ell
$$
and hence $\rho(m(\ell))=\rho(m)-\ell=a+1-\ell$,
one obtains from \cite[(2.4)]{Koorn84} and \cite[p. 90]{OpdamActa}  (for all values of the parameters for which the functions are well defined)
\begin{eqnarray}
\label{eq:Fell-rank1}
F_{\ell,\l}(m,x)&=&(\cosh x)^{-\ell} F_{\l}(m(\ell),x)=(\cosh x)^{-\ell} \varphi^{(a,-\ell)}_{i\l}(x) \notag \\
&=&(\cosh x)^{-\ell} \hyper{\frac{\rho(m)-\l-\ell}{2}}{\frac{\rho(m)+\l-\ell}{2}}{\rho(m)}{-\sinh^2 x}\,.\\
G_{\ell,\l}(m,x)&=&(\cosh x)^{-\ell} G_{\l}(m(\ell),x) \notag\\
&=&(\cosh x)^{-\ell} \Big[ \varphi^{(a,-\ell)}_{i\l}(x) +\frac{a+1+\l}{4(a+1)} \sinh(2x) \varphi^{(a+1,-\ell+1)}_{i\l}(x) \Big]
\notag \\
\label{eq:Gell-rank1}
&=&F_{\ell,\l}(m,x) + \frac{1}{4} \Big( 1+\frac{m_\rms/2 + 1 - \ell + \lambda}{m_\rms/2+1} \,   \Big) \sinh(2x)  F_{\ell,\l}(m+2\cdot 1_\rms,x)\,,
\end{eqnarray}
where $2 \cdot 1_\rms$ is the multiplicity function defined by $2 \cdot 1_\rms(\beta/2)=2$ and $2\cdot 1_\rms(\beta)=0$.
In particular, the classical relation $\hyper{\a}{\b}{\gamma}{z}=(1-z)^{\gamma-\a-\b}  \hyper{\gamma-\a}{\gamma-\b}{\gamma}{z}$ (see e.g. \cite[2.1.4(23)]{Er}) shows directly that $F_{\ell,\l}(m,x)=F_{-\ell,\l}(m,x)$ and hence, by \eqref{eq:Gell-rank1},
\begin{eqnarray*}
G_{-\ell,\l}(m,x)-G_{\ell,\l}(m,x)= \frac{\ell}{2(m_\rms/2+1)}\, \sinh(2x)  F_{\ell,\l}(m+2\cdot 1_\rms,x)\,.
\end{eqnarray*}

In the rank-one case, estimates for the hypergeometric functions
$F_{\l}(m(\ell))$  (for $\Re\l$ in a suitable neighborhood of $0$) can be obtained from some classical integral formulas for Gauss' hypergeometric function.
The hypergeometric identity
$$
\hyper{\a}{\b}{\gamma}{z}=(1-z)^{-\a}\hyper{\a}{\gamma-\b}{c}{\frac{z}{z-1}}
$$
(see \cite[2.1.4.(22)]{Er}) with $z =-\sinh^2 x$, yields the alternative expression
\begin{equation}
\label{Fell-rank1-5}
F_{\ell,\l}(m;x)=(\cosh x)^{\l-\rho(m)} \hyper{\frac{\rho(m)-\l-\ell}{2}}{\frac{\rho(m)-\l+\ell}{2}}{\rho(m)}{\tanh^2 x}\,.
\end{equation}

Let $B(x,y)=\frac{\Gamma(x)\Gamma(y)}{\Gamma(x+y)}$ denote Euler's beta function.
According to \cite[2.12.(1)]{Er},
\begin{equation}
\label{eq:intergralF-Erdelyi}
\hyper{\a}{\b}{\gamma}{z}=\frac{1}{B(\b,\gamma-\b)}\,
\int_0^1 u^{\b-1} (1-u)^{\gamma-\b-1} (1-uz)^{-\a} \; du
\end{equation}
provided $\Re \gamma >\Re \b >0$ and $|\arg(1-z)|<\pi$.
Applying this formula to \eqref{eq:Fell-rank1}, we obtain
\begin{equation}
\label{Fell-rank1-3}
F_{\ell,\l}(m;x)=\frac{1}{B\big(\frac{\rho(m)+\l-\ell}{2},\frac{\rho(m)-\l+\ell}{2}\big)} \int_0^1 u^{\frac{\rho(m)+\l-\ell}{2}-1} (1-u)^{\frac{\rho(m)-\l+\ell}{2}-1} (1+u\sinh^2 x)^{-\frac{\rho(m)+\l+\ell}{2}} \; du
\end{equation}
provided $-(\rho(m)-\ell)<\Re \l<\rho(m)+\ell$. If $\ell\geq 0$, this interval contains $0$ when $\ell<\rho(m)$. Under the same conditions on $\Re\l$,  the change of variables $t=\tanh^2 x$ gives the alternative expression
\begin{equation}
\label{Fell-rank1-4}
F_{\ell,\l}(m;x)=\frac{1}{B\big(\frac{\rho(m)+\l-\ell}{2},\frac{\rho(m)-\l+\ell}{2}\big)} \int_0^{+\infty} \frac{(\sinh t)^{\rho(m)-\l-1} (\cosh t)^{\l-\rho(m)-2\ell+1}}{(\cosh^2 t+\sinh^2 x \sinh^2 t)^{\frac{\rho(m)+\l}{2}}} \; dt\,
\end{equation}

\begin{Lemma}
\label{lemma:pos-rank1}
Suppose $\ell\in \R$ satisfies $|\ell|\leq \rho(m)=\frac{m_\rms}{2}+1$. Then $F_{\ell,\l}(m) \geq 0$ for all all $\l\in \R$.
\end{Lemma}
\begin{proof}
Since $F_{\ell,\l}=F_{-\ell,\l}$, it is enough to consider the case
$0\leq \ell\leq \rho(m)$. In this case, the result follows immediately from \eqref{Fell-rank1-3}.
\end{proof}


\section{Basic estimates and asymptotics}
\label{section:basic}

In this section we present some basic estimates of the
$\tau_{-\ell}-$Heckman-Opdam hypergeometric functions.
For geometric multiplicity functions and $\ell\in \Z$, the functions $F_{\ell,\l}(m)$ coincide with the restriction to $A\equiv \fa$ of the
$\tau_{-\ell}-$spherical functions $\varphi_{\ell,\l}$ for the line bundle $E_\ell$ over $G/K$. Their basic estimates, stated in Corollary \ref{Cor:estimates-spherical}, are consequence of the integral formula \eqref{eq:integral-formula}, which is missing for general $F_{\ell,\l}(m)$'s.

For arbitrary multiplicities, \eqref{eq:Fell-F} and
\eqref{eq:Gell-G} allow us to reduce the study of the functions
$F_{\ell,\l}(m)$ and $G_{\ell,\l}(m)$ to that of the usual
Heckman-Opdam symmetric and non-symmetric hypergeometric functions $F_{\l}(m(\ell))$ and $G_{\l}(m(\ell))$.

Let $\mathcal{M}$ denote the set of real-valued multiplicity functions $m=(m_\rms,m_\rmm,m_\rml)$ on a root system $\Sigma$ of type $BC_r$.
We will consider the following subsets of $\mathcal{M}$:
\begin{eqnarray}
\label{eq:M+}
\mathcal{M}_+&=&\{m\in \mathcal{M}:
\text{$m_\a\geq 0$ for every $\alpha\in\Sigma$}\}\,,\\
\label{eq:M0}
\mathcal{M}_0&=&\{(m_\rms,m_\rmm,m_\rml)\in \mathcal{M}: m_\rmm\geq 0, m_\rms+m_\rml\geq 0\}\,,\\
\label{eq:M1}
\mathcal{M}_1&=&\{(m_\rms,m_\rmm,m_\rml)\in \mathcal{M}:  m_\rmm> 0, m_\rms> 0, m_\rms+2m_\rml> 0\}\,,\\
\label{eq:M2}
\mathcal{M}_2&=&\{(m_\rms,m_\rmm,m_\rml)\in \mathcal{M}: m_\rmm\geq 0, m_\rml\geq 0, m_\rms+m_\rml\geq 0\}\,,\\
\label{eq:M3}
\mathcal{M}_3&=&\{(m_\rms,m_\rmm,m_\rml)\in \mathcal{M}: m_\rmm\geq 0, m_\rml\leq 0, m_\rms +2m_\rml\geq 0\}\,.
\end{eqnarray}
So $\mathcal{M}_+$ consists of the non-negative multiplicity functions
and $\mathcal{M}_1=(\mathcal{M}_+\cup \mathcal{M}_3)^0$, the interior of
$\mathcal{M}_+\cup \mathcal{M}_3$.
For real-valued multiplicities, $\mathcal{M}_0$ is the natural set for which both hypergeometric functions $G_\lambda(m)$ and $F_\lambda(m)$ are defined for all $\lambda\in\mathfrak{a}^*_\C$; see Appendix A.
In \cite[Definition 5.5.1]{HS} the elements of $\mathcal{M}_1$ are called standard multiplicity functions. These sets of multiplicities are represented in Figure \ref{fig:m}.
\begin{center}
\begin{figure}[h]
\begin{tikzpicture}[
    scale=1.8,
    axis/.style={very thick, ->, >=stealth'},
    equation line/.style={thin},
    ]
   \fill[gray!20!,path fading=east, fading angle=45]
    (0,0) -- (1.8,0) -- (1.8,1.8) -- (0,1.8) -- cycle;
   \fill[gray!110!,path fading=east]
   (0,0) -- (1.8,0) -- (1.8,-.9) -- cycle;
   \fill[gray!40!,path fading=east]
   (0,0) -- (1.8,-.9) -- (1.4,-1.4) -- cycle;
   \fill[gray!40!,path fading=north]
   (0,0) -- (0,1.8) -- (-1.4,1.4) -- cycle;
    \draw[axis] (-1.4,0)  -- (2,0) node(xline)[right]
        {$m_{\mathrm{s}}$};
    \draw[axis] (0,-1.4) -- (0,2) node(yline)[above]
        {$m_{\mathrm{l}}$};
     \draw[equation line] (0,0) -- (1.8,-.9)
        node[right, text width=10em, rotate=0]
        {$m_{\mathrm{s}}+2m_{\mathrm{l}}=0$};
     \draw[equation line] (-1.4,1.4) -- (1.4,-1.4)
        node[right, text width=10em, rotate=0]
        {$m_{\mathrm{s}}+m_{\mathrm{l}}=0$};
     \draw [->,line width=.5pt] (.8,-.4)
        arc[x radius=.9cm, y radius =.9cm,
        start angle=-22.5, end angle=88];
     \draw [->,line width=.5pt] (-.5,.5)
        arc[x radius=.7cm, y radius =.7cm,
        start angle=135, end angle=0];
    \draw (1.2,1.5) node[right] {$\mathcal{M}_+$};
  	\draw (-.67,.78) node[right] {{\footnotesize$\mathcal{M}_2$}};
  	\draw (.8,.5) node[right] {{\footnotesize $\mathcal{M}_1$}};
  	\draw (1.2,-.4) node[right] {$\mathcal{M}_3$};
\end{tikzpicture}
\caption{Sets of $BC$ multiplicities}
\label{fig:m}
\end{figure}
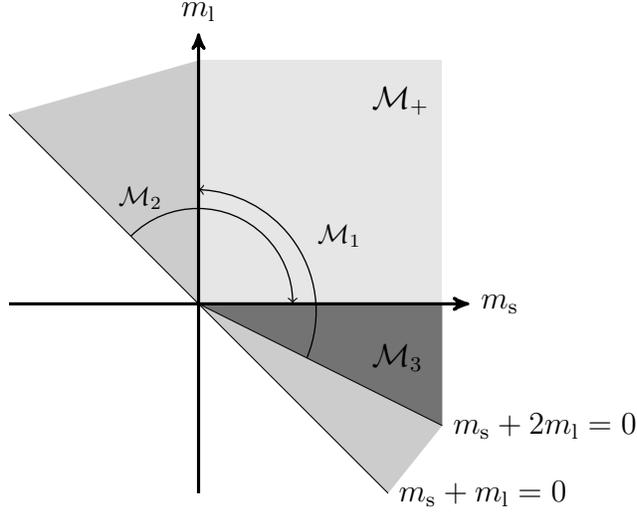
\end{center}

\subsection{Basic estimates}
\label{subsection:estimates}

Under the assumption that $m\in \mathcal{M}_+$, the positivity properties of  $F_{\l}(m)$ and $G_{\l}(m)$ for $\l\in\fa^*$ as well as their basic estimates have been proved by Schapira in \cite[\S 3.1]{Sch08}, by refining some ideas from \cite[\S 6]{OpdamActa}. Under the same assumption, Schapira's estimates for $F_\l(m)$ and $G_\l(m)$ have been sharpened by R\"osler, Koornwinder and Voit in \cite[\S 3]{RKV13}.
The following proposition collects their results.

\begin{Prop}
\label{prop:positivity}
Let $m\in \mathcal{M}_+$.
Then the following properties hold.
\begin{enumerate}
\thmlist
\item For all $\l\in\fa^*$ the functions $G_\l(m)$ and  $F_\l(m)$ are real and strictly positive.
\item For all $\l\in\fa_\C^*$
\begin{equation}
\label{eq:basic-estimate1}
\begin{split}
|G_\l(m)|&\leq G_{\Re\l}(m)\,,\\
|F_\l(m)|&\leq F_{\Re\l}(m)\,.
\end{split}
\end{equation}
\item For all $\l\in\fa^*$ and all $x\in \fa$
\begin{equation}
\begin{split}
\label{eq:basic-estimate2}
G_\l(m;x)&\leq G_0(m;x)e^{\max_w (w\l)(x)} \\
F_\l(m;x)&\leq F_0(m;x)e^{\max_w (w\l)(x)} \,.
\end{split}
\end{equation}
More generally, for all $\l\in\fa^*$, $\mu\in \overline{(\fa^*)^+}$ and all $x\in \fa$
\begin{equation}
\begin{split}
\label{eq:basic-estimate3}
G_{\l+\mu}(m;x)&\leq G_\mu(m;x)e^{\max_w (w\l)(x)} \\
F_{\l+\mu}(m;x)&\leq F_\mu(m;x)e^{\max_w  (w\l)(x)} \,.
\end{split}
\end{equation}
\end{enumerate}
(In the above estimates, $\max_w$ denotes the maximum over all $w\in W$.)
\end{Prop}

Below we will see that all these estimates extend to $m\in \mathcal{M}_3$. For this, we will need the following lemma.

\begin{Lemma}
\label{lemma:pos-coeffs}
Let $m=(m_\rms,m_\rmm,m_\rml)\in \mathcal{M}_0$ and $\beta\in \Sigma_{\rml}^+$. Then  the following inequalities hold for all $x\in \mathfrak{a}$:
\begin{enumerate}
\thmlist
\item
$\displaystyle{\frac{m_\rms}{2}+m_\rml \, \frac{1}{1+e^{-\beta(x)}}\geq 0}$\quad  if $m\in \mathcal{M_+}\cup \mathcal{M}_3$;
\smallskip

\item
$\displaystyle{\frac{m_\rms}{2}+m_\rml \, \frac{1+e^{-2\beta(x)}}{(1+e^{-\beta(x)})^2}\geq 0}$\quad  if $m\in \mathcal{M}_2\cup \mathcal{M}_3$.
\end{enumerate}
\end{Lemma}
\begin{proof}
If $m\in \mathcal{M_+}\cup \mathcal{M}_3$, then $m_\rms\geq 0$ and $m_\rml\geq -\frac{1}{2}m_\rms$. Hence, for every $0\leq C\leq 1$, we have
$\frac{m_\rms}{2}+C m_\rml  \geq (1-C)\frac{m_\rms}{2} \geq 0$.
We therefore obtain both (a) and (b) for these $m$'s by observing that for every $t=-\b(x)\in \R$ we have
$$0\leq \frac{1}{1+e^t}\leq 1 \qquad \text{and} \qquad 0\leq\frac{1+e^{2t}}{(1+e^t)^2}\leq 1\,.$$

Suppose now that $m\in \mathcal{M}_2$. Since $m_\rml\geq 0$, $m_\rms+m_\rml\geq 0$ and $\frac{1+|z|^2}{|1+z|^2}\geq \frac{1}{2}$ for all $z\in \C$, we immediately
obtain that for all $t=-\beta(x)\in \R$ we have
$$\frac{m_\rms}{2}+m_\rml \, \frac{1+e^{2t}}{(1+e^t)^2} \geq \frac{1}{2}(m_\rms+m_\rml) \geq 0\,.$$
This proves the lemma.
\end{proof}

\begin{Rem}
Since
$$
\lim_{t\to+\infty} \Big(\frac{m_\rms}{2}+m_\rml \frac{1}{1+e^t}\Big)=\frac{m_\rms}{2}
\quad \text{and} \quad
\left.\Big( \frac{1}{2} m_\rms+m_\rml\frac{1+e^{2t}}{(1+e^t)^2} \Big)\right|_{t=0}=\frac{1}{2}m_\rms+m_\rml\,,$$
it is clear that the inequality (a) of Lemma \ref{lemma:pos-coeffs} cannot extend to $m\in
\mathcal{M}_0 \setminus \big(\mathcal{M}_+\cup \mathcal{M}_3\big)$ and (b) cannot extend to $m\in
\mathcal{M}_0 \setminus \big(\mathcal{M}_2\cup \mathcal{M}_3\big)$.
\end{Rem}

To use the methods from \cite[\S 3]{RKV13}, we will
also need an extension of the real case of  Opdam's estimates, \cite[Proposition 6.1 (1)]{OpdamActa}. We will do this for $m\in \mathcal{M}_2\cup \mathcal{M}_3$.
Notice that the complex variable version of Opdam's estimates, \cite[Proposition 6.1 (2)]{OpdamActa}, has been recently extended by Ho and \'Olafsson \cite[Appendix A]{HoOl14} to $m\in \mathcal{M}_2$ and to a domain in $\mathfrak{a}_\C$ which is much larger than the original one considered by Opdam. For $m\in \mathcal{M}_2$  the inequality (b) of Lemma \ref{lemma:pos-coeffs} was noticed in the proof of \cite[Proposition 5.A, p. 25]{HoOl14}.

\begin{Lemma}
\label{lemma:OpdamM2M3}
Suppose that $m\in \mathcal{M}_2 \cup \mathcal{M}_3$.
Then for all $\l\in\fa_\C^*$ and $x\in \fa$
\begin{equation}
\label{eq:basic-estimate4}
\begin{split}
|G_\l(m;x)|&\leq \sqrt{|W|} e^{\max_w \Re(w\l)(x)}\,,\\
|F_\l(m;x)|&\leq \sqrt{|W|} e^{\max_w \Re(w\l)(x)}\,.
\end{split}
\end{equation}
\end{Lemma}
\begin{proof}
Regrouping the terms corresponding to
$\b$ and $\b/2$ and setting $z=x\in \mathfrak{a}$, we can
rewrite the first term on the right-hand side of the first displayed equation in \cite[p. 101]{OpdamActa} as
\begin{eqnarray*}
&&\sum_{w\in W, \a\in \Sigma_\rmm^+}
m_\a \frac{\a(\xi)(1-e^{-4\a(x)})}{(1-e^{-2\a(x)})^2} |\phi_w-\phi_{r_\a w}|^2 e^{-2\mu(x)}+\\
&&\quad \sum_{w\in W, \beta\in \Sigma_\rml^+}
\Big(m_{\beta/2} \frac{\frac{\b}{2}(\xi)(1-e^{-2\beta(x)})}{(1-e^{-\beta(x)})^2} +
m_{\beta} \frac{\beta(\xi)(1-e^{-4\beta(x)})}{|1-e^{-2\beta(x)})^2} \Big) |\phi_w-\phi_{r_\beta w}|^2 e^{-2\mu(x)}
\end{eqnarray*}
in which
$\alpha(\xi) (1-e^{-4\a(x)})$ is positive for all $\a\in \Sigma^+$ when $\xi\in \mathfrak{a}$ is chosen in the same Weyl chamber of $x$.
The coefficient of $ |\phi_w-\phi_{r_\beta w}|^2 e^{-2\mu(x)}$ for $\beta\in \Sigma_\rmm^+$ is then clearly positive. For $\beta\in \Sigma_\rml^+$, this coefficient is equal to
$$
\frac{\beta(\xi)(1-e^{-2\beta(x)})}{(1-e^{-\beta(x)})^2}
\Big(\frac{m_\rms}{2} +m_\rml
\frac{1+e^{-2\beta(x)}}{(1+e^{-\beta(x)})^2} \Big)\,,$$
which is positive by Lemma \ref{lemma:pos-coeffs},(b).
Therefore, the same proof as in \cite[Proposition 6.1]{OpdamActa} allows us to obtain the required inequalities.
\end{proof}

The following proposition extends the positivity properties and basic estimates from \cite[\S 3.1]{Sch08} and \cite[\S 3]{RKV13} to the multiplicity functions in $\mathcal{M}_3$.

\begin{Prop}
\label{prop:positivity-estimates-M3}
Proposition \ref{prop:positivity} holds for $m\in \mathcal{M}_3$.
\end{Prop}
\begin{proof}
The proofs of (a), (b) and \eqref{eq:basic-estimate2} follow the same steps as the proofs of \cite[Lemma 3.1 and Proposition
3.1]{Sch08}. So we shall just indicate what has to be
modified in the proofs when considering $m\in \mathcal{M}_3$ instead of $m\in \mathcal{M}_+$.

Suppose $G_\l(m)$ is not positive and let $x\in \fa$ be a zero of $G_\l(m)$ of minimal norm.
As in \cite[Lemma 3.1]{Sch08}, one has to distinguish whether $x$ is regular or singular. If $x$ is regular, take $\xi$ in the same chamber of $x$.
Evaluation at $x$ of the equation
\begin{equation}
\label{eq:hyperGlambda}
T_\xi(m) G_\l(m)=\l(\xi) G_\l(m)
\end{equation}
yields
\begin{multline*}
\partial_\xi G_\l(m;x)=\sum_{\a\in \Sigma^+} m_\alpha \frac{\a(\xi)}{1-e^{-2\a(x)}}
\big[G_\l(m;r_\a x)-G_\l(m;x)\big]
+\big(\rho(m)+\l\big)(\xi) G_\l(m;x)
\end{multline*}
in which $G_\lambda(m;x)$ vanishes.
In the sum over $\Sigma^+$, the coefficient of
$G_\l(m;r_\a x)-G_\l(m;x)$ is always non-negative for
$\a\in \Sigma_\rmm^+$. Moreover, grouping together
those corresponding to $\b/2$ and $\b\in \Sigma_\rml^+$, we obtain
as coefficient of
$G_\l(m;r_\b x)-G_\l(m;x)$
\begin{equation}
\label{eq:positivity-beta}
m_{\beta/2} \frac{\frac{\beta}{2}(\xi)}{1-e^{-\beta(x)}}
+m_{\beta} \frac{\beta(\xi)}{1-e^{-2\beta(x)}}
=\frac{\b(\xi)}{1-e^{-\b(x)}}\Big[ \frac{m_\rms}{2} +
 m_\rml \frac{1}{1+e^{-\b(x)}}\Big],
\end{equation}
which is positive by Lemma \ref{lemma:pos-coeffs},(a).

If $x$ is singular, let $I=\{\a \in \Sigma^+: \a(x)=0\}$ and let $\xi$ in the same face of $x$ so that $\alpha(\xi)=0$ for all $\alpha\in I$. In this case,
\eqref{eq:hyperGlambda} evaluated at $x$ gives
\begin{multline*}
\partial_\xi G_\l(m;x)=-\sum_{\a\in I} m_\a \frac{\a(\xi)}{\inner{\a}{\a}}\partial_{x_\a}  G_\l(m;x)\\
+\sum_{\alpha\in \Sigma_{\rmm}^+\setminus I}
m_{\rmm}  \frac{\a(\xi)}{1-e^{-2\a(x)}}
\big[G_\l(m;r_\a x)-G_\l(m;x)\big]\\
+\sum_{\b\in\Sigma_{\rml}^+ \setminus I}
\Big[m_{\beta/2} \frac{\frac{\beta}{2}(\xi)}{1-e^{-\beta(x)}}
+m_{\beta} \frac{\beta(\xi)}{1-e^{-2\beta(x)}}\Big]
\big[G_\l(m;r_\beta x)-G_\l(m;x)\big]\\
+\big(\rho(m)+\l\big)(\xi) G_\l(m;x)\,,
\end{multline*}
in which the first sum and $G_\l(m;x)$ vanish, and
the sum over $\Sigma^+_\rml\setminus I$ as coefficient as in \eqref{eq:positivity-beta}.

In both cases, one can argue as in
\cite[Lemma 3.1]{Sch08} by replacing the multiplicities $k_\a\geq 0$ in that proof with
$\frac{m_\rms}{2} +
 m_\rml\, \frac{1}{1+e^{-\b(x)}}$ and using Lemma \ref{lemma:pos-coeffs},(a).

Having that $G_{\Re \l}(m)$ is real and positive, to prove (b) for $G_\l(m)$, one can make the same grouping and substitution of multiplicities, as done above for $\partial_\xi G_\l(m;x)$, inside the formula for $\partial_\xi |Q_\l|^2(x)$ appearing in the proof of \cite[Proposition 3.1 (a)]{Sch08}. A third application of the same grouping in the formula of $\partial_\xi R_\l(x)$ in the proof of \cite[Proposition 3.1 (b)]{Sch08} yields \eqref{eq:basic-estimate2}.

Finally, \eqref{eq:basic-estimate3} has been proven for
multiplicities $m_\a\geq 0$ in \cite[Theorem 3.3]{RKV13} using the
original versions of (a), (b), (c) and  \eqref{eq:basic-estimate2}
together with a clever application of the Phragm\'en-Lindel\"of
principle that does not involve the root multiplicities. With Opdam's estimate for $m\in \mathcal{M}_3$, as in Lemma \ref{lemma:OpdamM2M3}, the original proof from \cite{RKV13}
extends to the case of $m\in \mathcal{M}_3$ and yields \eqref{eq:basic-estimate3}.
\end{proof}


\subsection{Asymptotics}
\label{subsection:asymptotics}

We now investigate the asymptotic behavior of
the hypergeometric functions $F_\l(m)$ for $m\in \mathcal{M}_{1}$.
For this we follow \cite{NPP}. Before we state our results, it is useful to recall
the methods adopted in \cite{NPP}. Let $m$ be an arbitrary
non-negative multiplicity function. One of the main results in
\cite{NPP} is Theorem 2.11, where a series expansion away from the
walls was obtained for $F_\lambda(m) $ for all $\lambda \in
\frakacs$ (even when $\lambda$ is non-generic). Recall from \cite[\S 4.2]{HS} (or
from \cite[\S 1.2]{NPP} in symmetric space notation) that, for
generic $\lambda$, the function $F_\lambda(m)$ is given on $\fa^+$ by
$\sum_{w \in W} c(m;w\lambda) \Phi_{w\lambda}(m)$, where
$c(m;\lambda)$ is Harish-Chandra's $c$-function (see \eqref{eq:c} in the appendix)  and
$\Phi_\lambda(m;x)$ admits the series expansion
$$\Phi_\lambda(m;x) = e^{(\lambda-\rho(m))(x)} \sum_{\mu \in 2\Lambda}
\Gamma_\mu(m;\lambda) e^{-\mu(x)} \qquad (x\in \fa^+).$$
The coefficients
$\Gamma_\mu(m,\lambda)$ are determined from the recursion relations
$$ \inner{\mu}{\mu-2\lambda} \Gamma_\mu(m,\lambda)= 2 \sum_{\alpha \in
\Sigma^+} m_\alpha \sum_{n \in \mathbb N, \mu-2n\alpha \in
\Lambda} \Gamma_{\mu-2n\alpha}(m,\lambda) \inner
{\mu+\rho(m)-2n\alpha-\lambda}{\alpha} ,$$ with the initial
condition that $\Gamma_0(m)= 1$.
The above defines
$\Gamma_\mu(m;\lambda)$ as meromorphic functions on $\frakacs.$

For a non-generic point $\lambda = \lambda_0$ a series expansion for
$F_\lambda(m)$  was obtained in \cite{NPP} following the steps given
below:

\medskip
\textit{Step I:} A listing of the possible singularities of the
$c$-function and the coefficients  $\Gamma_\mu(m;\lambda)$ at $\lambda =
\lambda_0.$

\smallskip
\textit{Step II:} Identify a polynomial $p$ so that $$\lambda \to
p(\lambda) \Big ( \sum_{w \in W}
c(m;w\lambda)e^{(\lambda-\rho(m))(x)} \sum_{\mu \in 2\Lambda}
\Gamma_\mu(m;\lambda) e^{-\mu(x)} \Big )$$ is holomorphic in a
neighborhood of $\lambda_0.$

\smallskip
\textit{Step III:} Write $F_{\lambda_0}(m) = a~\partial(\pi)
(pF_\lambda(m))|_{\lambda = \lambda_0}$ where $\partial(\pi)$ is the
differential operator corresponding to the highest degree
homogenous term in $p$ and $a,$ is a non-zero constant. This gives
the series expansion of $F_{\lambda_0}(m).$

\medskip
A careful examination shows that the same proofs go through even
with the assumption that the multiplicity function belongs to
$\mathcal{M}_{1}$.
Indeed, the possible singularities of
the $c$-function and the $\Gamma_\mu(m)$ are contained in the same set of
hyperplanes as listed in \cite[Lemma 2.3]{NPP}. Hence the same
polynomials and differential operators can be used in Step II and
III above. The crucial detail to be checked is the
computation of the constant $b_0(m;\lambda_0)$ appearing in
\cite[Lemma 2.6]{NPP}.

Let $\lambda_0 \in \frakacs$ with $\Re
\lambda_0 \in \overline{(\fa^\ast)^+},$
Its explicit expression in terms of
Harish-Chandra's $c$-function shows that, for root systems of type $BC_n\setminus C_n$, the function $b_0(m;\lambda_0)$ is
non-zero if and only if
\begin{equation}
\label{eq:b-nonzero}
\prod_{\alpha \in \Sigma^+_\rms}
\Gamma\Big(\frac{(\lambda_0)_\alpha}{2} +\frac{m_\rms}{4}+\frac{1}{2}\Big)
\Gamma\Big(\frac{(\lambda_0)_\alpha}{2} +\frac{m_\rms}{4}+\frac{m_\rml}{2}\Big)
\prod_{\alpha \in \Sigma^+_\rmm}
\Gamma\Big(\frac{(\lambda_0)_\alpha}{2} +\frac{m_\rmm}{4}+\frac{1}{2}\Big) \Gamma\Big(\frac{(\lambda_0)_\alpha}{2} +\frac{m_\rmm}{4}\Big)
\end{equation}
is nonsingular.
This is clearly true if $m\in \mathcal{M}_{1}$.
(Notice that the nonvanishing of $b_0(m;\lambda_0)$ identifies the main term in the expansion of
$F_{\lambda_0}(m;x)$ as
$\frac{b_0(m;\lambda_0)}{\pi_0(\rho_0(m))} \pi_0(x) e^{(\lambda_0 - \rho(m))(x)}$,
where $\rho_0(m)$ is defined as in \cite[(58)]{NPP}.
Notice also that $c(m;\lambda_0)$, and hence $b_0(m;\lambda_0)$, can vanish
for $m\in \mathcal{M}_0\setminus \mathcal{M}_1$.)

It follows from the above that \cite[Theorem 2.11]{NPP} and, as a
consequence, \cite[Theorem 3.1]{NPP} continue to hold true for a
multiplicity function $m\in \mathcal{M}_{1}$.
We state it below and refer to \cite{NPP} for any unexplained notation.

\begin{Thm}
\label{thm:hc-expansion}
Suppose $m\in \mathcal{M}_{1}$. Let $\lambda_0 \in \frakacs$ with $\Re
\lambda_0 \in \overline{(\fa^\ast)^+},$ and let $x_0 \in \fa^+$ be
fixed. Then, there are constants $C_1 > 0, C_2 > 0$ and $b > 0$
(depending on $m$, $\lambda_0$ and $x_0$) so that for all $x \in x_0 +
\fa^+ :$
\begin{multline} \label{eq:restF-est}
\Big| \frac{F_{\l_0}(m; x)
e^{-(\Re\l_0-\rho(m))(x)}}{\pi_0(x)} -\Big(
\frac{b_0(m;\l_0)}{\pi_0(\rho_0(m))} e^{i\Im\l_0(x)} + \sum_{w\in
W_{\Re\l_0} \setminus W_{\l_0}} \! \!
\frac{b_w(m;\l_0)\pi_{w,\l_0}(x)}{c_0\pi_0(x)}
e^{i w\Im\l_0(x)} \Big)\Big|  \\
\leq C_1 (1+\b(x))^{-1} +C_2 (1+ \b(x))^{|\Sigma_{\l_0}^+|}
e^{-b\beta(x)}\,,
\end{multline}
where $\beta(x)$ is the minimum of $\alpha(x)$ over the simple roots $\alpha\in \Sigma^+$
and the term $C_1 (1+\b(x))^{-1}$ on the right-hand side of
(\ref{eq:restF-est}) does not occur if $\inner{\alpha}{\l_0}\neq 0$ for all $\alpha\in \Sigma$.
\end{Thm}

Notice that, for fixed $x_0\in \fa^+$, we have $\beta(x) \asymp |x|$ as $x\to \infty$ in $ x_0+\overline{\fa^+}$, where $|x|$ is the Euclidean norm on $\fa$.

Likewise, one can extend to $m\in \mathcal{M}_{1}$ the following corollary, which restates Theorem \ref{thm:hc-expansion} in the special case where $\l_0\in \overline{(\frak a^*)^+}$.

\begin{Cor} \label{cor:leading-termF-realcase}
Suppose $m\in \mathcal{M}_{1}$.  Let $\l_0 \in \overline{(\frak a^*)^+}$, and let $x_0 \in \frak
a^+$ be fixed. Then there are constants $C_1>0$, $C_2>0$ and $b>0$
(depending on $m$, $\l_0$ and $x_0$) so that for all $x \in x_0+
\overline{\frak a^+}$:
\begin{multline} \label{eq:restF-est-realcase}
\Big| F_{\l_0}(m; x) -\frac{b_0(m;\l_0)}{\pi_0(m;\rho_0)} \pi_0(x) e^{(\l_0-\rho(m))(x)}\Big|  \leq \\
\leq \big[ C_1 (1+\b(x))^{-1} +C_2 (1+ \b(x))^{|\Sigma_{\l_0}^+|}
e^{-b\beta(x)}\big] \pi_0(x) e^{(\l_0-\rho(m))(x)}\,.
\end{multline}
The term $C_1 (1+\b(x))^{-1}$ on the right-hand side of
(\ref{eq:restF-est-realcase}) does not occur if $\inner{\alpha}{\l_0}\neq 0$ for all $\alpha\in \Sigma$.
\end{Cor}

It might be useful to observe that for $m\in \mathcal{M}_0$ the only obstruction to \eqref{eq:restF-est} is that $b_0(m;\lambda_0)\neq 0$.

\begin{Cor}
\label{cor:asymptotics-b0}
Let $m\in \mathcal{M}_0$. Then Theorem 4.6 holds for every $\l_0\in \mathfrak{a}_\C^*$ with $\Re\l_0\in\overline{(\mathfrak{a}^*)^+}$ such that
\eqref{eq:b-nonzero} is nonsingular.
\end{Cor}

\subsection{Sharp estimates}
\label{subsection:sharp}

In this subsection we assume that $m\in \mathcal{M}_3$.
Let $\mathcal{M}_3^0$ denote the interior of $\mathcal{M}_3$.
Since $\mathcal{M}_3^0 \subset \mathcal{M}_1$, the results of both subsections \ref{subsection:estimates} and \ref{subsection:asymptotics} are available on $\mathcal{M}_3^0$.

Let $\l\in \frak a^*$. Using the nonsymmetric hypergeometric
functions $G_\l(m)$, their relation to the hypergeometric
function, and the systems of differential-reflection equations they satisfy, Schapira proved in \cite{Sch08} the following local Harnack principle for the hypergeometric function $F_{\l}(m)$: for all $x \in \overline{\frak a^+}$
\begin{equation} \label{eq:localHarnack}
\nabla F_{\l}(m;x)=-\frac{1}{|W|} \sum_{w \in W}
w^{-1}(\rho(m)-\l)G_{\l}(m;wx)\,,
\end{equation}
the gradient being taken with respect to the space variable $x\in
\frak a$. It holds for every multiplicity function $m$ for which both $G_\l(m)$ and $F_\l(m)$ are defined. See \cite[Lemma 3.4]{Sch08}. Since
$\partial_\xi F=\inner{\nabla F}{\xi}$ and since $G_\l(m)$ and $F_\l(m)$ are real and non-negative for $m\in \mathcal{M}_3$, one obtains as in \textit{loc. cit.}  that for all $\xi \in \frak a$
$$
\partial_\xi \Big( e^{K_\xi \frac{\inner{\xi}{\cdot}}{|\xi|^2}} F_{\l}(m;\cdot) \Big) \leq 0\,,
$$
where $K_\xi=\max_{w\in W} (\rho(m)-\l)(w\xi)$. This in turn yields
the following subadditivity property, which is implicit in
\cite{Sch08} for $m\in \mathcal{M}_+$ and in fact holds also for $m\in \mathcal{M}_3^0$ and, by continuity, on $\mathcal{M}_3$.

\begin{Lemma}\label{lemma:subadd-Schapira}
Suppose $m\in \mathcal{M}_3^0$.
Let $\l\in\frak a^*$. Then for all $x, x_1 \in \frak a$ we have
\begin{equation} \label{eq:subadd-Schapira}
F_{\l}(m; x+x_1) e^{-\max_{w\in W} (\l-\rho(m))(wx_1)}
\leq F_\l(m; x) \leq F_{\l}(m; x+x_1) e^{\max_{w\in W}
(\rho(m)-\l)(wx_1)}\,.
\end{equation}
In particular, if $\l \in \overline{(\frak a^*)^+}$ and $x_1 \in
\frak a^+$, then
\begin{equation} \label{eq:subadd-Schapira-special}
F_{\l}(m; x+x_1) e^{-(\l+\rho(m))(x_1)} \leq
F_{\l}(m; x) \leq F_{\l}(m; x+x_1)
e^{(\l+\rho(m))(x_1)}\,
\end{equation}
for all $x \in \frak a$.
\end{Lemma}

Together with Corollary \ref{cor:leading-termF-realcase}, the
above lemma yields the following global estimates of
$F_{\l}(m; x).$

\begin{Thm}\label{thm:real}
Let $m\in \mathcal{M}_3$ and $\l_0 \in \overline{(\frak a^*)^+}$. Then for all $x \in
\overline{\frak a^+}$ we have
\begin{equation}
F_{\l_0}(m; x) \asymp \big[\prod_{\a\in\Sigma_{\l_0}^0}
(1+\a(x))\big] e^{(\l_0-\rho(m))(x)}\,,
\end{equation}
where $\Sigma_{\l_0}^0=\{ \a\in \Sigma^+_\rms\cup \Sigma^+_\rmm:\inner{\a}{\l_0}=0\}$.
\end{Thm}
\begin{proof}
Same as in \cite[Theorem 3.4]{NPP}.
\end{proof}

\subsection{The case of 1-parameter deformed multiplicities}
\label{subsection:estimates-mell}

For $m\in \mathcal{M}_0$ and $\ell\in \R$ we set
\begin{equation}
\label{eq:mult-l-gen}
m_\a(\ell)=\begin{cases}
m_\rms+2\ell  &\text{if $\a\in \Sigma_\rms$}\\
m_\rmm &\text{if $\a\in \Sigma_\rmm$}\\
m_\rml-2\ell  &\text{if $\a\in \Sigma_\rml$}\,.
\end{cases}
\end{equation}
(This extends \eqref{eq:mult-l} by dropping the condition $m_\rml=1$.) Notice that $m(\ell)\in \mathcal{M}_0$ because $m_\rms(\ell)+m_\rml(\ell)=m_\rms+m_\rml$.
Recall that in this paper we consider the root system $C_r$ as a root system
$BC_r$ with $m_\rms=0$.

For a fixed $m=(m_\rms,m_\rmm,m_\rml)$ we shall use the notation
\begin{equation}
\label{eq:lmin-lmax}
\ell_{\rm min}(m)=-\frac{m_\rms}{2}\qquad \text{and}\qquad  \ell_{\rm max}(m)=\frac{m_\rms}{2}+m_\rml\,.
\end{equation}
We simply write $\ell_{\rm min}$ and $\ell_{\rm max}$
when this does not cause any ambiguity.

\begin{Thm}
\label{thm:M+M3}
Let $m^0=(m^0_\rms, m^0_\rmm,m^0_\rml)\in \mathcal{M}_0$. Then  $m^0\in \mathcal{M}_+ \cup \mathcal{M}_3$ if and only if there are
$m=(m_\rms,m_\rmm,m_\rml)\in \mathcal{M}_+$ and
$\ell\in \R$ so that $m^0=m(\ell)$ and
$\ell \in \big[\ell_{\rm min}(m),\ell_{\rm max}(m)\big]$.
Moreover, $m^0\in \mathcal{M}_1=(\mathcal{M}_+ \cup \mathcal{M}_3^0)^0$ if and only if  $m=(m_\rms,m_\rmm,m_\rml)\in \mathcal{M}_+$ is as above and
$\ell \in \big]\ell_{\rm min}(m),\ell_{\rm max}(m)\big[$.
\end{Thm}
\begin{proof}
If $m^0\in \mathcal{M}_+ \cup \mathcal{M}_3$, then
$m^0=m(\ell)$ for
$$
m=(m_\rms,m_\rmm,m_\rml)=(m^0_\rms+m^0_\rml,m^0_\rmm,0)\in \mathcal{M}_+ \quad\text{and} \quad \ell=-\frac{m^0_\rml}{2}\,.$$
Observe that $\ell$ satisfies
$-\frac{m_\rms}{2}=-\frac{m^0_\rms+m^0_\rml}{2}\leq \ell \leq \frac{m^0_\rms+m^0_\rml}{2}=\frac{m_\rms}{2}+m_\rml$
because $m^0_\rms\geq 0$ and $m^0_\rms+2m^0_\rml\geq 0$.
The inequalities for $\ell$ are strict if $m^0 \in \mathcal{M}_1$ since in this case
$m^0_\rms>0$.

Conversely, suppose that $m^0=m(\ell)$ for $m\in \mathcal{M}_+$ and $\ell$ as in the statement.
Then $-m_\rms\leq 2\ell \leq m_\rms+2m_\rml$ and
$m_\rms+m_\rml=m^0_\rms+m^0_\rml$.
Hence
$$m^0_\rms=m_\rms+2\ell\leq 2(m_\rms+m_\rml)=2(m^0_\rms+m^0_\rml), \quad \text{i.e.} \quad m^0_\rms+2m^0_\rml\geq 0\,.$$ Moreover,
$m^0_\rms=m_\rms+2\ell\geq m_\rms-m_\rms=0$.
Thus $m^0 \in \mathcal{M}_+\cup \mathcal{M}_3$.
All the inequalities are strict if $\ell \in \big]\ell_{\rm min}(m),\ell_{\rm max}(m)\big[$. Hence, in this case, $m^0 \in \mathcal{M}_1$.
\end{proof}

For $(m_\rms,m_\rml)=(m_\rms,1)$ geometric, the values of $(m_\rms(\ell),m_\rml(\ell)) =(m_\rms+2\ell,1-2\ell)$
with $\ell\in \big[\ell_{\rm min},\ell_{\rm max}]$ are
represented in Figure 2.

\begin{figure}
\label{fig:mell}
\begin{tikzpicture}[
    scale=.5,
    axis/.style={very thick, ->, >=stealth'},
    equation line/.style={thick},
    equation line dashed/.style={thick, dashed},
    reference line/.style={thin},
    ]
\draw [decorate sep={.8mm}{10mm},fill] (0,1) -- (10,1);
\draw [decorate sep={.5mm}{10mm},fill] (0,0) -- (20,0);
\draw [decorate sep={.4mm}{3mm},fill] (2.5,9.5) -- (5,12);
\draw [decorate sep={.4mm}{3mm},fill] (8.5,3.5) -- (11,6);
\draw [decorate sep={.4mm}{3mm},fill] (16.5,-4.5) -- (19,-2);
\draw [dotted] (2,0) -- (2,1);
\draw [dotted] (4,0) -- (4,1);
\draw [dotted] (6,0) -- (6,1);
\draw [dotted] (8,0) -- (8,1);
\draw [dotted] (10,0) -- (10,1);
\draw [dotted] (12,0) -- (12,1);
\draw [dotted] (14,0) -- (14,1);
\draw [dotted] (16,0) -- (16,1);
\draw [dotted] (18,0) -- (18,0);
\draw [dotted] (20,0) -- (20,0);

\draw[fill=black] (16,1) circle (.8mm);

\draw[equation line dashed] (4,13) -- (6,11);
\draw[equation line dashed] (19,-2) -- (21,-4);

    \draw[axis] (-.5,0)  -- (22,0) node(xline)[right]
        {$m_{\mathrm{s}}$};
    \draw[axis] (0,-.5) -- (0,14.5) node(yline)[above]
        {$m_{\mathrm{l}}$};
     \draw[reference line] (0,0) -- (20,-10)
        node[below left, text width=8em, rotate=-28]
        {$m_{\mathrm{s}}+2m_{\mathrm{l}}=0$};
     \draw[reference line] (0,1) -- (22,1)
        node[above, text width=10em, rotate=0]
        {\null\qquad $m_{\mathrm{l}}=1$};
     \draw[equation line] (0,1) -- (2,-1);
     \draw[equation line] (0,3) -- (6,-3);
     \draw[equation line] (0,5) -- (10,-5);
     \draw[equation line] (0,7) -- (14,-7);
     \draw[equation line] (0,9) -- (18,-9);
     \draw[equation line] (0,11) -- (20,-9);
    \draw[equation line] (6,11) -- (19,-2);

\node[left, rotate=-45] at (0,1) {$\Sigma=C_r$ \;};
\node[left, rotate=-45] at (0,3)
        {$\mathrm{SU}(p,p+1)$ \;};
\node[left, rotate=-45] at (0,5)
        {$\mathrm{SO}^*(2(2n+1))$
        \& $\mathrm{SU}(p,p+2)$ \;};
\node[left, rotate=-45] at (0,7)
        {$\mathrm{SU}(p,p+3)$ \;};
\node[left, rotate=-45] at (0,9)
        {$\mathfrak{e}_{6(-14)}$
        \& $\mathrm{SU}(p,p+4)$ \;};
\node[left, rotate=-45] at (0,11)
        {$\mathrm{SU}(p,p+5)$ \;};
\node[left, rotate=-45] at (4,13)
        {$\mathrm{SU}(p,q)$ \;};

\node at (14,8) {$\mathcal{M}_+$};
\node at (19,-5) {$\mathcal{M}_3$};

\draw (-.4,-.2) node[below,scale=.6] {$0$};
  	\draw (2,-.2) node[below,scale=.6] {$2$};
  	\draw (4,-.2) node[below,scale=.6] {$4$};
  	\draw (6,-.2) node[below,scale=.6] {$6$};
  	\draw (8,-.2) node[below,scale=.6] {$8$};
  	\draw (10,-.2) node[below,scale=.6] {$10$};
  	\draw (16,-.2) node[below,scale=.6] {$2(q-p)$};

\node[right, scale=1] at (7,15) {%
\parbox{8truecm}
{
\begin{tabular}{|c|c|c|c|}
\hline
$G$ & $m_\rms$ & $\ell_{\rm min}$
& $\ell_{\rm max}$\\
\hline
with $\Sigma=C_r$ & 0 & 0 & 1\\
\hline
$\SO^*(2(2n+1))$ & 4 & $-2$ & 3\\
\hline
$\mathfrak{e}_{6(-14)}$ & 8 & $-4$ & 5\\
\hline
$\SU(p,q)$ $q>p$ & $2(q-p)$ & $p-q$ & $q-p+1$\\
\hline
\end{tabular}
}};
\end{tikzpicture}
\caption{$(m_\rms(\ell),m_\rml(\ell))$ for geometric
$(m_\rms,m_\rml=1)$}
\end{figure}
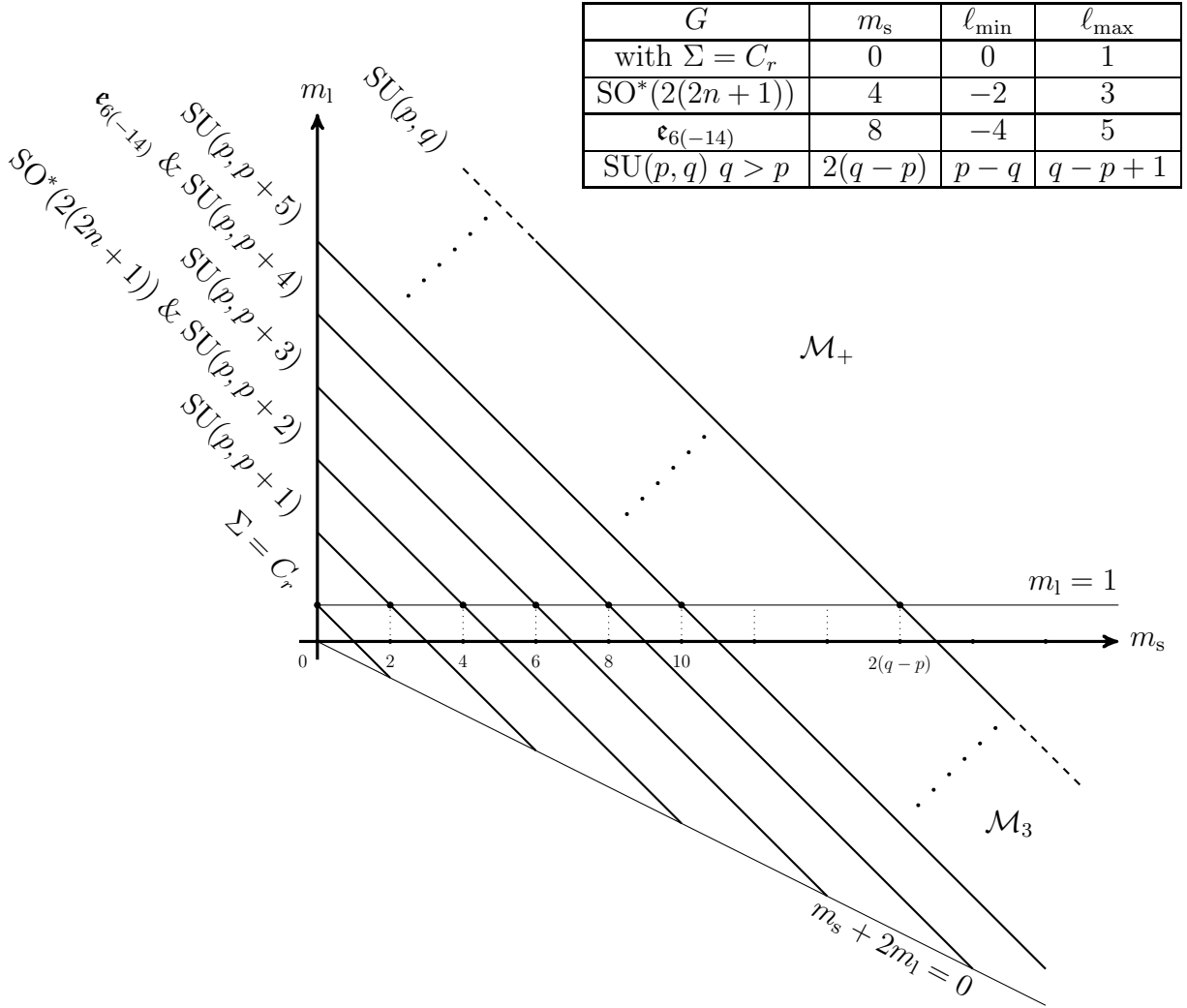

\subsection{Consequences for the $\tau_{-\ell}$-hypergeometric functions}
\label{subsection:taul}
Let $m=(m_\rms,m_\rmm,m_\rml=1)\in \mathcal{M}_+$.
Using \eqref{eq:Fell-F} and \eqref{eq:Gell-G}, we can obtain from the results of this sections some estimates and asymptotic properties for the $\tau_{-\ell}$-hypergeometric functions. By Theorem \ref{thm:M+M3}, we have $m(\ell)\in
\mathcal{M}_+\cup \mathcal{M}_3$ for $m\in \mathcal{M}_+$ and
$\ell\in \big[\ell_{\rm min}(m),\ell_{\rm max}(m)\big]$.

This implies that Theorem \ref{thm:hc-expansion} and
Corollary \ref{cor:leading-termF-realcase} hold true for
$F_\l(m(\ell))$ for all $m\in \mathcal{M}_+$ and $\ell\in \big]\ell_{\rm min}(m),\ell_{\rm max}(m)\big[$. In turn, \eqref{eq:Fell-F} yields analogous statements for the $\tau_{-\ell}$-hypergeometric functions.

Recall that $\ell_{\rm min}=-\frac{m_\rms}{2}$ and $\ell_{\rm max}=\frac{m_\rms}{2}+m_\rml=\frac{m_\rms}{2}+1$. (We are omitting from the notation the dependence on $m$ of $\ell_{\rm min}$ and $\ell_{\rm max}$). So
$\ell_{\rm min} \leq 0 < \ell_{\rm max}=|\ell_{\rm min}|+1$.

Recalling also that $F_{\ell,\lambda}=F_{-\ell,\lambda}$, we see that we can extend the inequalities for $F_{\ell,\lambda}$ to $|\ell|< \ell_{\rm max}=\frac{m_\rms}{2}+1$ (or to $|\ell|\leq \ell_{\rm max}$ where extension by continuity in the multiplicity parameter is possible; see Theorem \ref{thm:reg-prop} in the appendix).
We leave to the reader the simple task of modifying the statements of Theorem \ref{thm:hc-expansion} and
Corollary \ref{cor:leading-termF-realcase} in this case, and
we collect the estimates obtained for $G_{\ell,\lambda}(m)$ and $F_{\ell,\lambda}(m)$ in the corollary below.
Recall that $\ell_{\rm min}=-\frac{m_\rms}{2}$ and $\ell_{\rm max}=\frac{m_\rms}{2}+m_\rml$. So
$\ell_{\rm min} \leq 0 < \ell_{\rm max}=|\ell_{\rm min}|+1$.

\begin{Cor}
\label{cor:pos-est-ell}
Let $m=(m_\rms,m_\rmm,m_\rml=1) \in \mathcal{M}_+$ be a nonnegative multiplicity function
on a root system of type $BC_r$, and let $\ell\in\R$. Then the following properties hold.
\begin{enumerate}
\thmlist
\item For all $\ell \leq\ell_{\rm max}$ and  $\l\in\fa_\C^*$
\begin{equation}
\begin{split}
\label{eq:b-e1}
|G_{\ell,\l}(m)|&\leq \sqrt{|W|} e^{\max_w (w\l)(x)}\,, \\
|F_{\ell,\l}(m)|&=|F_{-\ell,\l}(m)|\leq \sqrt{|W|} e^{\max_w (w\l)(x)}\,.
\end{split}
\end{equation}
\item For all $\ell \in [\ell_{\rm min},\ell_{\rm max}]$ and
$\l\in\fa^*$ the functions $G_{\ell,\l}(m)$ and  $F_{\ell,\l}(m)=F_{-\ell,\l}(m)$ are real and strictly positive.
\item For all $\ell \in [\ell_{\rm min},\ell_{\rm max}]$ and    $\l\in\fa_\C^*$
\begin{equation}
\begin{split}
\label{eq:b-e1}
|G_{\ell,\l}(m)|&\leq G_{\ell,\Re\l}(m) \,,\\
|F_{\ell,\l}(m)|&\leq F_{\ell,\Re\l}(m)\,.
\end{split}
\end{equation}
\item For all $\ell \in [\ell_{\rm min},\ell_{\rm max}]$, $\l\in\fa^*$ and all $x\in \fa$
\begin{equation}
\begin{split}
\label{eq:b-e2}
G_{\ell,\l}(m;x)&\leq G_{\ell,0}(m;x)e^{\max_w (w\l)(x)}\,, \\
F_{\ell,\l}(m;x)&\leq F_{\ell,0}(m;x)e^{\max_w (w\l)(x)} \,.
\end{split}
\end{equation}
More generally, for all $\ell \in [\ell_{\rm min},\ell_{\rm max}]$, $\l\in\fa^*$, $\mu\in \overline{(\fa^*)^+}$ and all $x\in \fa$
\begin{equation}
\label{eq:b-e3}
\begin{split}
G_{\ell,\l+\mu}(m;x)&\leq G_{\ell,\mu}(m;x)e^{\max_w (w\l)(x)}\,, \\
F_{\ell,\l+\mu}(m;x)&\leq F_{\ell,\mu}(m;x)e^{\max_w (w\l)(x)} \,.
\end{split}
\end{equation}
\item
Suppose $|\ell| \leq \ell_{\rm max}$, $\l \in \overline{(\frak a^*)^+}$ and $x_1 \in
\frak a^+$. Then for all  $x \in \frak a$ we have
\begin{equation} \label{eq:subadd-Schapira-special}
F_{\ell,\l}(m; x+x_1) e^{-(\l+\rho(m(\ell)))(x_1)} \leq
F_{\ell,\l}(m; x) \leq F_{\ell,\l}(m; x+x_1)
e^{(\l+\rho(m(\ell)))(x_1)}\,.
\end{equation}
\item
Suppose $|\ell| < \ell_{\rm max}$ and $\l_0 \in \overline{(\frak a^*)^+}$. Then for all $x \in
\overline{\frak a^+}$ we have
\begin{eqnarray}
F_{\ell,\l_0}(m; x) &\asymp & u^{-\ell}(x)\big[\prod_{\a\in\Sigma_{\l_0}^0}
(1+\a(x))\big] e^{(\l_0-\rho(m(\ell)))(x)}\\
&\asymp & \big[\prod_{\a\in\Sigma_{\l_0}^0}
(1+\a(x))\big] e^{(\l_0-\rho(m))(x)}\,,
\end{eqnarray}
where $\Sigma_{\l_0}^0=\{ \a\in \Sigma_\rms^+\cup \Sigma_\rmm^+:\inner{\a}{\l_0}=0\}$.
\end{enumerate}
\end{Cor}

\section{Bounded $\tau_{-\ell}-$hypergeometric functions and applications}

In this section we address the problem of characterizing the
$\tau_{\ell}-$hypergeometric functions which are bounded.
In particular, this also include characterizing the boundedness of the $\tau_{\ell}-$spherical functions studied by Shimeno \cite{ShimenoPlancherel}.
Recall that we are considering a multiplicity function $m = (m_\rms, m_\rmm, m_\rml=1)\in \mathcal{M}_+$ and its
deformation $m(\ell)$ as in \eqref{eq:mult-l}. To apply the estimates from subsection \ref{subsection:taul}, we will assume that $|\ell|\leq \ell_{\rm max}=\frac{m_\rms}{2}+1$. The asymptotics from Theorem 4.6 hold under the stronger assumption that $|\ell|< \ell_{\rm max}=\frac{m_\rms}{2}+1$.

Recall also that
$$\rho(m(\ell))=\frac{1}{2} \sum_{\a \in \Sigma^+} m_\a(\ell) \a=
 \rho(m)-\frac{\ell}{2}\, \sum_{j=1}^r \beta_j .$$
Notice that, if we suppose that $\ell \in [0,\ell_{\rm max}]$, then $m(\ell)\in \mathcal{M}_+\cup \mathcal{M}_3$ by Theorem \ref{thm:M+M3}. Hence  $\rho(m(\ell))\in \overline{(\mathfrak{a}^*)^+}$.

\begin{Thm}
\label{thm:bdd}
Fix $\ell \in \R$ with $|\ell|< \ell_{\rm max}$. Then, the $\tau_{-\ell}-$hypergeometric
function $F_{\ell, \l}(m)$ is bounded if and only if $\l \in
C(\rho(m)) + i \fa^\ast,$ where $C(\rho(m))$ is the convex hull of
the set $\{w\rho(m):~w \in W\}.$ Moreover, $|F_{\ell, \l}(m;x)| \leq 1$ for all $\l \in
C(\rho(m)) + i \fa^\ast$ and $x \in \fa$.
\end{Thm}
\begin{proof}
Since $F_{\ell, \l} = F_{-\ell, \l}$ we may assume that $\ell \geq 0.$ First we show that $F_{\ell, \l}$ is bounded if $\l \in
C(\rho(m))+i \fa^\ast.$  Let $\lambda \in C(\rho(m)) + i\fa^\ast$
and consider the holomorphic function $\l \to F_{\ell, \l}(m ; x)$
(for a fixed $x$). Since $$ |F_{\ell, \l}(m ; x)| \leq F_{\ell,
\Re \l}(m ; x) = u^{-\ell}(x) F_{\Re \l}(m(\ell); x),$$ we can
argue as in the proof of \cite[Theorem 4.2]{NPP} to obtain that
the maximum of $|F_{\ell, \l}(m; x)|$ is attained at
$\{w\rho(m):~w \in W\}.$ That is
$$
|F_{\ell, \l}(m; x)| \leq u^{-\ell}(x) F_{\rho(m)}(m(\ell); x).
$$
As noticed before,  $\rho(m(\ell))\in \overline{(\fa^\ast)^+}.$
Applying \eqref{eq:basic-estimate3}, we obtain
$$
F_{\rho(m)}(m(\ell); x) =
F_{\rho(m(\ell)) + \frac{\ell}{2} \sum_{j=1}^r \beta_j}(m(\ell);
x) \leq F_{\rho(m(\ell))}(m(\ell) ; x)e^{\max_w w \frac{\ell}{2}
\sum_{j=1}^r \beta_j(x)}.
$$
Now, $F_{\rho(m(\ell))}(m(\ell); x) =1$; see e.g. \cite[Lemma 4.1]{NPP} (the proof extends to every multiplicity function for which the symmetric and non-symmetric hypergeometric functions are defined).
Choose $x_0$ in the $W$-orbit of $x$ so that
$$
\max_w w \frac{\ell}{2}
\sum_{j=1}^r \beta_j(x) = \frac{\ell}{2} \sum_{j=1}^r
\beta_j(x_0).
$$
Since $u$ is $W-$invariant, from the above we get
$$
|F_{\ell, \l}(m; x)| \leq u^{-\ell}(x_0)
e^{ \frac{\ell}{2}\sum_{j=1}^r \beta_j(x_0)} \leq 1.
$$

To prove the
other way, we proceed as in \cite{NPP} (see pages 251--252). Let
$\lambda_0$ be such that $\Re \lambda_0 \in
\overline{(\fa^\ast)^+} \setminus C(\rho(m)).$ Let $x_1 \in \fa^+$
be such that $(\Re \lambda_0 - \rho(m))(x_1) > 0.$ If $F_{\ell,
\lambda_0}(m)$ is bounded we have $$ \lim_{t \to \infty} F_{\ell,
\lambda_0}(m;tx_1) e^{t(-\Re \lambda_0 -\rho(m))}(x_1) t^{-d} = 0.$$
Since, $F_{\ell \lambda_0}(m;x)  = u^{-\ell} F_{\lambda_0}(m(\ell);
x)$ and $u^{-\ell}(tx_1)$ behaves like $\prod_{j=1}^r
e^{-\frac{\ell}{2} t \beta_j(x_1)}$ for large $t$ we have
$$ F_{\lambda_0}(m(\ell); tx_1) e^{t(-\Re \lambda_0 - \rho(m(\ell)))(x_1)}
t^{-d} \to 0 \qquad t \to \infty.$$
So, using Theorem \ref{thm:hc-expansion}, the proof can be completed as in Theorem 4.2 of \cite{NPP} when $0\leq \ell< \ell_{\rm max}$.
\end{proof}

\begin{Rem}
The proof of Theorem \ref{thm:bdd} shows that the $\tau_{-\ell}-$hypergeometric
function $F_{\ell, \l}(m)$ is bounded, with $|F_{\ell, \l}(m;x)| \leq 1$ for all  $x \in \fa$ and $\l \in C(\rho(m)) + i \fa^\ast$, provided $|\ell| \leq\ell_{\rm max}$.
For general root systems of type $BC_r$, we cannot prove that when
$|\ell|=\ell_{\rm max}$ the function $F_{\ell, \l}(m)$ is not bounded for $\l \notin
C(\rho(m)) + i \fa^\ast$. This turns out to be true when $r=1$.
Indeed, for $\ell=\ell_{\rm max}$ we have $m_\rms(\ell)+2m_\rml(\ell)=0$ and, by \eqref{eq:b-nonzero},  the function $b_0(m(\ell_{\rm max});\lambda_0)$ vanishes for $\lambda_0 \in \frakacs$ with $\Re \lambda_0 \in \overline{(\fa^\ast)^+},$ if and only if there is $j\in \{1,\dots,r\}$ such that
$(\lambda_0)_{\beta_j}=0$. Outside these $\lambda_0$, the asymptotics from Theorem \ref{thm:hc-expansion} still hold. In the rank-one case, there is
one only $\lambda_0$ for which the asymptotics do not hold, namely $\lambda_0=0$ which is not outside $C(\rho(m)) + i \fa^\ast$. So the same proof allows us to conclude.
\end{Rem}

\begin{Cor}
Let $G/K$ be a Hermitian symmetric space and let
$|\ell|<\ell_{\rm max}=\frac{m_\rms}{2}+1$. The $\tau_{-\ell}-$spherical function $\varphi_{\ell, \l}(m)$ on $G/K$ is bounded if and only if $\lambda \in C(\rho(m)) + i \fa^\ast,$ where $C(\rho(m))$ is the convex hull of
the set $\{w\rho(m):~w \in W\}.$ Moreover, $|\varphi_{\ell, \l}(m;x)| \leq 1$ for all $\l \in
C(\rho(m)) + i \fa^\ast$ and $x \in \fa$.
\end{Cor}
\begin{proof}
We can suppose that $\ell\geq 0$ so that $m(\ell)\in \mathcal{M}_1$. This corollary is then an immediate consequence of Corollary \ref{Cor:estimates-spherical},(3) together with the second part of Theorem \ref{thm:bdd}.

\end{proof}

\appendix

\section{The Heckman-Opdam hypergeometric functions as functions of their parameters}
\label{appendix:A}

Let $\fa$ be an $r$-dimensional Euclidian real vector space, with an inner product
$\inner{\cdot}{\cdot}$, and let $\Sigma$ be a root system $\Sigma$ in $\fa^*$ of Weyl group $W$.
Let $\mathcal M_\C$ denote the set of complex-valued multiplicity functions $m=\{m_\alpha\}$ on $\Sigma$.
(Hence $\mathcal M_\C\equiv \C^d$ where $d$ is the number of Weyl group orbits in $\Sigma$.)

In this appendix we collect the regularity properties of the (symmetric and nonsymmetric) hypergeometric
functions $F_\l(m;x)$ and $G_\l(m;x)$ as functions of $(m,x,\l)\in \mathcal M_\C\times \fa_\C \times \fa_\C^*$.
Most of the results are known, but scattered in the litterature.

Recall the Gindikin-Karpelevich formula for Harish-Chandra's $c$-function:
\begin{equation}\label{eq:c}
c(m;\l)=\frac{\wt{c}(m;\l)}{\wt{c}(m;\rho(m))}  \qquad (\l\in\fa_\C^*)\,,
\end{equation}
where $\wt{c}(m;\l)$ is given in terms of the positive indivisible roots $\a \in \Sigma_{\rm i}^+$ by
\begin{equation} \label{eq:wtc}
\wt{c}(m;\l)=\prod_{\a \in \Sigma_{\rm i}^+} \frac{2^{-\la} \; \Gamma(\la)}
{\Gamma\Big(\frac{\la}{2}+\frac{m_\a}{4}+\frac{1}{2}\Big)
\Gamma\Big(\frac{\la}{2}+\frac{m_\a}{4}+\frac{m_{2\a}}{2}\Big)}
\end{equation}
and $\Gamma$ is the classical gamma function. Set (see \cite[p. 196]{OpdamIV})
$$
\mathcal{M}_{F, {\rm reg}}=\{m \in \mathcal{M}_\C: \text{$\frac{1}{\wt{c}(m;\rho(m))}$ is not singular}\}
$$
For an irreducible representation $\delta\in \widehat{W}$ and $m\in \mathcal{M}_\C$, let
$\varepsilon_\delta(m)=\sum_{\a\in\Sigma^+} m_\a (1-\chi_\delta(r_\alpha)/\chi_\delta(\id))$,
where $r_\alpha$ is the reflection across $\Ker(\alpha)$ and $\chi_\delta$ is the character of $\delta$.
Let $d_\delta$ be the lowest embedding degree of $\delta$ in $\C[\fa_\C]$, and
set (see \cite[Definition 3.13]{OpdamActa})
$$
\mathcal{M}_{G, {\rm reg}}=\{m \in \mathcal{M}: \text{$\Re(\varepsilon_\delta(m))+d_\delta>0$ for all $\delta\in
\widehat{W}\setminus \{{\rm triv}\}$}\}\,.
$$
Finally, let
$$
\Omega=\{x \in \fa: \text{$|\alpha(x)|<\pi$ for all $\alpha\in \Sigma^+$}\}\,.
$$
The regularity properties of the (symmetric and nonsymmetric) functions are summarized in the following theorem.

\begin{Thm}
\label{thm:reg-prop}
The hypergeometric function $F_\lambda(m;x)$ is  holomorphic in $(m,x,\l)\in
\mathcal{M}_{F, {\rm reg}} \times (\fa+i\Omega) \times \fa_\C^*$ and satisfies
$$F_{\l}(m;wx) = F_{\l}(m;x)=F_{w\l}(m;x) \qquad (m \in \mathcal{M}_{F, {\rm reg}}, \, w\in W, \, x \in \fa+i\Omega, \l\in \fa_\C^* )\,.$$
The (non-symmetric) hypergeometric function $G_\lambda(m;x)$ is a holomorphic function of $(m,x,\l)\in
\mathcal{M}_{G, {\rm reg}} \times (\fa+i\Omega) \times \fa_\C^*$.
\end{Thm}
\begin{proof}
For $\fa+i\Omega$ replaced by a $W$-invariant tubular domain $V$ of $\fa$ in $\fa_\C$, these results are due to Opdam, see \cite[Theorem 2.8]{OpdamIV} and \cite[Theorem 3.15]{OpdamActa}. See also \cite[Theorem 4.4.2]{HS}. The remark that the maximal tubular domain $V$ is $\fa+i\Omega$ was made by Jacques Faraut, see \cite[Remark 3.17]{BOP}.
\end{proof}

\begin{Prop}
Set
\begin{eqnarray}
\label{eq:MC+}
\mathcal{M}_{\C,+}&=&\{m\in \mathcal{M}_\C: \text{$\Re(m_\a)\geq 0$ for every $\alpha\in \Sigma^+$}\}\\
\label{eq:MC0}
\mathcal{M}_{\C,0}&=&\{m\in \mathcal{M}_\C: \text{$\Re(m_\a+m_{2\a})\geq 0$ for every $\alpha\in \Sigma_{\rm i}^+$}\}\,.
\end{eqnarray}
Then
$$\mathcal{M}_{\C,+}\subset \mathcal{M}_{\C,0} \subset
\mathcal{M}_{F, {\rm reg}} \cap \mathcal{M}_{G, {\rm reg}}\,.$$
Moreover, $\mathcal{M}_{\C,0}$ is stable under deformations by $\ell\in \C$ as in \eqref{eq:mult-l-gen}:
$m\in \mathcal{M}_{\C,0}$ if and only if $m(\ell)\in \mathcal{M}_{\C,0}$.
\end{Prop}
\begin{proof}
The inclusion $\mathcal{M}_{\C,0} \subset \mathcal{M}_{F, {\rm reg}} $ was observed in \cite[Remark 4.4.3]{HS} and
follows from the properties of the Gamma function. Notice that $\Re(\rho(m)_\alpha)\geq \Re(\frac{m_\a}{2}+m_{2\a})$ for $\a\in \Sigma_{\rm i}^+$, with simplifications in the factor of $c(m;\rho(m))$ corresponding to $\a$ in case of
equality with $\frac{m_\a}{2}+m_{2\a}$ real.
For the inclusion  $\mathcal{M}_{\C,0} \subset \mathcal{M}_{G, {\rm reg}}$, observe that
$\varepsilon_\delta(m)=\sum_{\a\in\Sigma_{\rm i}^+} (m_\a+m_{2\a}) (1-\chi_\delta(r_\alpha)/\chi_\delta(\id))$,
with $\chi_\delta(r_\alpha)/\chi_\delta(\id))\in [-1,1]$.
\end{proof}



\begin{thebibliography}{22}
\bibitem{BOP}
Branson, T. , \'Olafsson, G. and A. Pasquale:
The Paley-Wiener theorem and the local Huygens' principle for compact symmetric spaces: the even multiplicity case, \textit{Indag. Math. (N.S.)} \textbf{16} (2005), no. 3--4, 393--428.

\bibitem{CM82}
Casselman, W. and D. Mili{\v{c}}i{\'c}:
Asymptotic behavior of matrix coefficients of admissible representations.
\textit{Duke Math. J.} \textbf{49} (1982), no. 4, 869-–930.

\bibitem{Er}
Erd\'elyi, A. et~al.:
{\em {Higher transcendental functions}}, volume~1, McGraw-Hill, New York, 1953.

\bibitem{GV}
Gangolli, R. and V.~S.~Varadarajan:
{\em Harmonic analysis of spherical functions on real reductive groups}.
Springer-Verlag, Berlin, 1988.

\bibitem{HC56}
Harish-Chandra:
The characters of semisimple Lie groups. \textit{Trans. Amer. Math. Soc.}
\textbf{83} (1956), 98--163.

\bibitem{HC58}
\bysame:
Spherical Functions on a Semisimple Lie Group, I.
\textit{Amer. J. Math.} \textbf{80} (1958), 241--310.

\bibitem{HC60}
\bysame: Differential equations and semisimple Lie groups.
In:
\textit{Harish-Chandra. Collected papers.} Vol. III. 1959–-1968. Edited by V. S. Varadarajan. Springer-Verlag, New York, 1984, 57--120.

\bibitem{HeckBou}
Heckman, G.:
Dunkl operators.
{\em Ast\'erisque} \textbf{245} (1997), Exp.\ No.\ 828, 4, 223--246.
S\'eminaire Bourbaki, Vol.\ 1996/97.


\bibitem{HS}
Heckman, G. and H.~Schlichtkrull:
\newblock {\em {Harmonic analysis and special functions on symmetric spaces}}.
\newblock Academic Press, 1994.

\bibitem{He1}
Helgason, S.:
\textit{Differential geometry, Lie groups, and symmetric spaces}.
Academic Press, 1978.

\bibitem{HeJo}
Helgason, S. and K. Johnson:
The bounded spherical functions on symmetric spaces. \emph{Advances in Math.} \textbf{3} (1969), 586--593.

\bibitem{He2}
\bysame:
{\em Groups and geometric analysis. Integral geometry, invariant
  differential operators, and spherical functions,}
Mathematical Surveys and Monogrpahs, 83. AMS, Providence, RI,
2000.

\bibitem{HoOl14}
Ho, V. M. and G. \'Olafsson:
\textit{Paley-Wiener Theorem for Line Bundles over Compact Symmetric Spaces},
preprint, 2014. arXiv:1407.1489

\bibitem[21]{Koorn84}
Koornwinder, T.~H.:
Jacobi functions and analysis on noncompact semisimple {L}ie groups.
In {\em Special functions: group theoretical aspects and
applications}, pages 1--85. Reidel, Dordrecht, 1984.

\bibitem{KW65}
Kor\'anyi, A. and J. Wolf: Realization of hermitian symmetric spaces as generalized half-planes. \textit{Ann. of Math. (2)} \textbf{81} (1965), 265--288.

\bibitem{NPP}
Narayanan, E. K., Pasquale, A. and S. Pusti: Asymptotics of
Harish-Chandra expansions, bounded hypergeometric functions
associated with root systems, and applications. \textit{Advances
Math.} \textbf{252} (2014), 227-259


\bibitem{Oda14}
Oda, H.:
\textit{Differential-difference operators and radial part formulas for non-invariant elements},
preprint, 2014, arxiv:1402.3231


\bibitem{OpdamIV}
Opdam, E.:
Root systems and hypergeometric functions. IV.
\textit{Compos. Math.} \textbf{67} (1988), no. 2, 191--209.

\bibitem{OpdamActa}
\bysame:
Harmonic analysis for certain representations of graded Hecke algebras.
{\em Acta math.} \textbf{175} (1995), 75--12.

\bibitem{Opd00}
\bysame:
\newblock {\em Lecture notes on {D}unkl operators for real and complex
  reflection groups}.
\newblock Mathematical Society of Japan, Tokyo, 2000.

\bibitem{RKV13}
R\"osler, M.,  Koornwinder, T. and M. Voit:
Limit transition between hypergeometric functions of type $BC$ and type $A$.
\textit{Compos. Math.} \textbf{149} (2013), no. 8, 1381--1400.

\bibitem{Sch08}
Schapira, B.: \textit{Etude analytique et probabiliste de laplaciens associ\'es \`a des syst\`emes de racines : laplacien hyperg\'eom\'etrique de Heckman--Opdam et laplacien combinatoire sur les immeubles affines.} Th\`ese,  Universit\'e d'Orl\'eans, 2006. Available at
https://hal.archives-ouvertes.fr/tel-00115557

\bibitem{SchThese}
\bysame: Contributions to the hypergeometric function theory
of Heckman and Opdam: sharp estimates, Schwartz space, heat
kernel. {\em Geom. Funct. Anal.} \textbf{18} (2008), no. 1,
222--250.

\bibitem{Schl84}
Schlichtkrull, H.:
One-dimensional K-types in finite-dimensional representations of semisimple Lie groups: a generalization of Helgason's theorem.
\textit{Math. Scand.} \textbf{54} (1984), no. 2, 279--294.

\bibitem{ShimenoEigenfunctions}
Shimeno, N.:  Eigenspaces of invariant differential operators on a homogeneous line bundle on a Riemannian symmetric space. \textit{J. Fac. Sci. Univ. Tokyo Sect. IA Math.} \textbf{37} (1990), 201--234.

\bibitem{ShimenoPlancherel}
\bysame: The Plancherel formula for spherical functions with a one-dimensional K-type on a simply connected simple Lie group of Hermitian type. \textit{J. Funct. Anal.} \textbf{121} (1994), no. 2, 330--388.

\bibitem{ShimenoBC}
\bysame: A formula for the hypergeometric function of type $BC_n$. \textit{Pacific J. Math.} \textbf{236} (2008), no. 1, 105--118.

\end{thebibliography}
\end{document}